\documentclass[12pt]{amsart}

\usepackage{amstext    }
\usepackage{amsthm    }
\usepackage{a4}
\usepackage[mathscr]{eucal}
\usepackage{mathrsfs}
\usepackage{hyperref}

\usepackage{amsmath}
\usepackage{amssymb}
\usepackage{amscd}

 \usepackage{amsmath,amstext,amsthm,amscd,typearea,hyperref}
\usepackage{amssymb}
\usepackage{a4wide}
\usepackage[mathscr]{eucal}
\usepackage{mathrsfs}
\usepackage{typearea}
\usepackage{charter}
\usepackage{pdfsync}

\usepackage{amscd,amsxtra,calc}
\usepackage{cmmib57}
\usepackage{url}

\newtheorem{theorem}{Theorem}[section]
\newtheorem{definition}[theorem]{Definition}
\newtheorem{proposition}[theorem]{Proposition}
\newtheorem{corollary}[theorem]{Corollary}
\newtheorem{lemma}[theorem]{Lemma}
\newtheorem{remark}[theorem]{Remark}

\newtheorem{example}[theorem]{Example}

\newcommand{\cali}[1]{\mathscr{#1}}

\newcommand{\Leb}{{\rm Leb}}

\newcommand{\Har}{{\rm Har}}

\newcommand{\dist}{{\rm dist}}

\newcommand{\dbar}{{\overline\partial}}
\newcommand{\ddbar}{{\partial\overline\partial}}

\newcommand{\GL}{{\rm GL}}

\newcommand{\id}{{\rm id}}
\newcommand{\pr}{{\rm pr}}

\newcommand{\hol}{{\rm hol}}

\newcommand{\Mb}{\mathop{\mathbf{M}}\nolimits}
\newcommand{\mb}{\mathop{\mathbf{m}}\nolimits}

\newcommand{\esup}{{\rm ess.\ sup}}

\newcommand{\bfp}{{\rm \bf p}}

\newcommand{\Ac}{\cali{A}}
\newcommand{\Bc}{\cali{B}}
\newcommand{\Cc}{\cali{C}}

\newcommand{\Ec}{\cali{E}}

\newcommand{\Lc}{\cali{L}}
\renewcommand{\Mc}{\cali{M}}

\newcommand{\Sc}{\cali{S}}

\newcommand{\C}{\mathbb{C}}
\newcommand{\D}{\mathbb{D}}
\newcommand{\K}{\mathbb{K}}

\newcommand{\N}{\mathbb{N}}
\newcommand{\Z}{\mathbb{Z}}
\newcommand{\R}{\mathbb{R}}
\newcommand{\T}{\mathbb{T}}
\newcommand{\B}{\mathbb{B}}
\newcommand{\U}{\mathbb{U}}

\renewcommand{\P}{\mathbb{P}}

\newcommand{\E}{\mathbb{E}}

\title{Geometric characterization of Lyapunov exponents for  Riemann surface   laminations}

\author{ Vi{\^e}t-Anh Nguy{\^e}n}

\dedicatory{Dedicated to the memory of Gennadi M. Henkin}

\begin{document}

\maketitle

\begin{abstract}
  We     characterize  geometrically  the  Lyapunov exponents of a   cocycle (of arbitrary  rank) with respect to 
a    harmonic  current defined  on a  hyperbolic  Riemann surface  lamination. Our  characterizations are formulated
in terms of  the expansion rates of the  cocycle  along geodesic rays.
%
  \end{abstract}

\noindent
{\bf Classification AMS 2010:} Primary: 37A30, 57R30;  Secondary: 58J35, 58J65, 60J65.



 \section{Introduction} \label{intro}

 
 The present work is motivated  by the  interplay  between geometry, topology and  dynamics   in the theory of hyperbolic Riemann  surface laminations. 
This class of laminations has been  extensively  studied by numerous  authors from  different viewpoints and using  various methods.
The  reader is invited  to  consult  the surveys by  Forn\ae ss-Sibony \cite {FornaessSibony2} and 
by Ghys \cite{Ghys} as  well as the references therein
for  a recent account on this  subject.

To start with  we  fix  some notation and terminology. We refer the  reader to Definition \ref{D:hyperbolic} below for the notion of  hyperbolic  Riemann  surface laminations.
Throughout this  work $\D$   denotes the unit disc in $\C$ and
  $g_P$ is  the Poincar\'e metric on $\D,$ given by
$$ g_P(\zeta):={2\over (1-|\zeta|^2)^2} \;i d\zeta\wedge d\overline\zeta,\qquad\zeta\in\D, \qquad\text{where}\qquad i:=\sqrt{-1}.  $$
Let $(X,\Lc)$ be   a hyperbolic  Riemann  surface lamination. We emphasize that   $X$ is not necessarily compact.
For any point $x\in X,$  let $L_x$  be the   leaf passing  through $x$ and  consider a universal covering map
\begin{equation}\label{e:covering_map}
\phi_x:\ \D\rightarrow L_x\qquad\text{such that}\  \phi_x(0)=x.
\end{equation}
 This map is
uniquely defined by $x$ up to a rotation on $\D$. 
Then, by pushing   forward  the Poincar\'e metric $g_P$
on $\D$  
  via $\phi_x,$ we obtain the  so-called {\it Poincar\'e metric} on $L_x$ which depends only on the leaf.  The latter metric is given by a positive $(1,1)$-form on $L_x$  that we also denote by $g_P$ for the sake of simplicity.  A subset $M\subset X$  is  called
  {\it leafwise  saturated} if $x\in M$ implies  $L_x\subset M.$

 To  the lamination $(X,\Lc)$ we  associate  several objects of different  nature.
 On the  dynamical  side, consider  the  sample-path space $\Omega$  which describes the leafwise Brownian motion. Namely,
 let  $\Omega:=\Omega(X,\Lc) $  be  the space consisting of  all continuous  paths  $\omega:\ \R^+:=[0,\infty)\to  X$ with image fully contained  in a  single   leaf. 
Consider  the  semi-group $(\sigma_t)_{t\in\R^+}$ of shift-transformations 
  $\sigma_t:\  \Omega\to\Omega$ defined for  all $t,s\in\R^+$ by 
$$   \sigma_t(\omega)(s):=\omega(s+t),\qquad  \omega\in \Omega.$$
For $x\in X,$ let  $\Omega_x$ be the  subspace  consisting of all  paths $\omega$ in $\Omega$ starting from $x,$ i.e., $\omega(0)=x.$ 
We endow $\Omega_x$ with a canonical  probability measure: the {\it Wiener measure} $W_x$  with respect  to the  metric $g_P$ on $L_x$ (see Subsection \ref{ss:Wiener} below).

 As  objects of   topological nature, we  deal  with  (multiplicative)  cocycles  which have been   introduced 
in a previous work \cite{NguyenVietAnh1}.  Prior to their  formal definition,   
 we make  the following convention.   Throughout the article,  $\K$ denotes either $\R$ or $\C.$
Moreover,  given any  integer $d\geq 1,$  $\GL(d,\K)$   denotes  the general linear group of degree $d$ over $\K$
and $\P^d(\K)$ denotes the  $\K$-projective space of dimension $d.$
\begin{definition}\label{D:cocycle}\rm
 A {\it $\K$-valued   cocycle} (of rank $d$)   is
  a   map  
$\mathcal{A}:\ \Omega\times \R^+ \to  \GL(d,\K)      $
such that\\  
(1)  ({\it identity law})  
$\mathcal{A}(\omega,0)=\id$  for all $\omega\in\Omega ;$\\
(2) ({\it homotopy law}) if  $\omega_1,\omega_2\in \Omega_x$ and $t_1,t_2\in \R^+$ such that 
     $\omega_1(t_1)=\omega_2(t_2)$
and $\omega_1|_{[0,t_1]}$ is  homotopic  to  $\omega_2|_{[0,t_2]}$ (that is, the path $\omega_1|_{[0,t_1]}$ can be  deformed  continuously on  $L_x$ to the path  $\omega_2|_{[0,t_2]},$ the two endpoints of $\omega_1|_{[0,t_1]}$  being kept fixed  during the deformation), then 
$$
\mathcal{A}(\omega_1,t_1)=\mathcal{A}(\omega_2,t_2);
$$
(3) ({\it multiplicative law})    $\mathcal{A}(\omega,s+t)=\mathcal{A}(\sigma_t(\omega),s)\mathcal{A}(\omega,t)$  for all  $s,t\in \R^+$ and $\omega\in \Omega;$\\
(4) ({\it measurable law})  the {\it local expression} of $\mathcal{A}$ on each  laminated  chart is   Borel measurable
(see  Subsection \ref{ss:background} below  for the  definition of  local  expressions).
 \end{definition}
 It is  worthy noting that the  cocycles of rank $1$  have been  investigated  by several  authors (see, for example, Candel \cite{Candel2}, Deroin \cite{Deroin}, etc).   
 The  holonomy cocycle (or equivalently the normal  derivative cocycle)  of the  regular part of a $n$-dimensional (possibly singular) holomorphic  foliation  by hyperbolic Riemann  surfaces  provides a typical   example of  $\C$-valued cocycles of rank $n-1.$
Another  source of cocycles  are those associated  with foliations  which are obtained  from suspensions. These   cocycles   capture  the topological aspect of the considered  foliations. Moreover, 
we can produce  new  cocycles from the old ones by performing some basic operations such as the wedge  product and the tensor product
(see \cite[Section 3.1]{NguyenVietAnh1}). 
In  this  article,  we  are mainly concerned  with cocycles which  behave in a tempered manner 
relative to the   metric $g_P.$ More concretely, we will introduce in Definition
 \ref{D:moderate_cocycle} below   two large families of cocycles: the  moderate  cocycles and 
 the   H\"older ones. 
 
 The objects of   geometric nature  considered  here are       {\it harmonic   currents} given by Garnett \cite{Garnett}  which are  generalizations of the  {\it foliations cycles} previously 
introduced by  Sullivan  \cite{Sullivan}. When $X$ is compact, the  existence  of     non-zero  harmonic currents $T$ has been  established by
 Garnett \cite{Garnett}. The case when $X$ is  non compact  has  been  investigated by  Berndtsson-Sibony and  Forn\ae ss-Sibony \cite{BerndtssonSibony,FornaessSibony2}. 
To   a   non-zero  harmonic current $T$  we  associate a Borel positive  measure
\begin{equation}\label{E:harmonic_measure}
\mu=\mu_T:=T\wedge g_P,
\end{equation}
which is  also  a {\it harmonic  measure} in  good  cases (for  example, when $X$ is  compact, see Proposition \ref{P:moderate_criterion} below).
 The notion  of harmonic  measures  and harmonic  currents  will be  recalled in    Definition  \ref{D:harmonic_measure} and \ref{D:harmonic_current}.
The  following terminology  will be  repeatedly used in this  article.
Given a  positive finite measure space $(S,\Sc,\nu),$  a  set $A\in \Sc$  is  called {\it of full $\nu$-measure}
if $\nu(S\setminus A)=0.$
In what follows, 
we say that a  set $A\subset X$ is  {\it of full $T$-measure} for a harmonic  current $T$ if it is  of full $\mu$-measure,
where $\mu$ is  given in  (\ref{E:harmonic_measure}).



 
 In this  setting, using our    recent  work  \cite{NguyenVietAnh1} 
 we  obtain the   following Oseledec   multiplicative ergodic theorem
which relates different objects of  different natures.   
  \begin{theorem} \label{T:VA}  
 Let $(X,\Lc)$ be  a  $\Cc^2$-smooth hyperbolic Riemann surface lamination and $T$  a harmonic current. 
 Let 
  $\mu$ be the measure associated to $T$ by (\ref{E:harmonic_measure}).
Assume, moreover,  that  $T$ is  extremal  in the cone of all   harmonic  currents and that $\mu$  is a (finite)  harmonic measure.
Consider  a   moderate cocycle
$\mathcal{A}:\ \Omega\times \R^+ \to  \GL(d,\K)  .    $  
Then  there exist  a leafwise saturated  Borel  set $Y\subset X$ of  full $\mu$-measure  and a number $m\in\N$  together with $m$ integers  $d_1,\ldots,d_m\in \N$  such that
the following properties hold:
\begin{itemize}
\item[(i)] For   each $x\in Y$  
 there   exists a  decomposition of $\K^d$  as  a direct sum of $\K$-linear subspaces 
$$\K^d=\oplus_{i=1}^m H_i(x),
$$
 such that $\dim H_i(x)=d_i$ and  $\mathcal{A}(\omega, t) H_i(x)= H_i(\omega(t))$ for all $\omega\in  \Omega_x$ and $t\in \R^+.$   
Moreover,  $x\mapsto  H_i(x)$ is   a  measurable map from $  Y $ into the Grassmannian of $\K^d.$
For each $1\leq i\leq m$ and each $x\in Y,$ let $V_i(x):=\oplus_{j=i}^m H_j(x).$  Set $V_{m+1}(x)\equiv \{0\}.$
\item[
(ii)]  There   are real numbers 
$$\chi_m<\chi_{m-1}<\cdots
<\chi_2<\chi_1,$$
  and   for  each $x\in Y,$ there is a set $F_x\subset \Omega_x$ of full $W_x$-measure such that for every $1\leq i\leq m$ and  every  $v\in V_i(x)\setminus V_{i+1}(x)$
 and every  $\omega\in F_x,$
\begin{equation}
\label{e:Lyapunov}
\lim\limits_{t\to \infty, t\in \R^+} {1\over  t}  \log {\| \mathcal{A}(\omega,t)v   \|\over  \| v\|}  =\chi_i.    
\end{equation}
Moreover, 
\begin{equation}
\label{e:Lyapunov_max}
\lim\limits_{t\to \infty, t\in \R^+} {1\over  t}  \log {\| \mathcal{A}(\omega,t)  \|}  =\chi_1    
\end{equation}
 for  each $x\in Y$  and for   every  $\omega\in F_x.$ 
\end{itemize}
Here    $\|\cdot\|$  denotes the standard   Euclidean norm of $\K^d.$  
 \end{theorem}  
The  above  result  is  the counterpart, in the context of hyperbolic Riemann surface laminations,
of the classical  Oseledec   multiplicative ergodic theorem for maps   (see \cite{KatokHasselblatt,Oseledec}).
  
 Assertion  (i) above says that the  Oseledec  decomposition exists for all points $x$ in  a leafwise saturated Borel  set
of  full $\mu$-measure and that this  decomposition is holonomy invariant.
It is  worthy noting that the Oseledec  decomposition  in  (i)  depends only on $x\in Y,$ in particular, it 
does not depend on paths $\omega\in\Omega_x.$

The decreasing  sequence  of  subspaces  of $\K^d$ given by assertion (i):
$$
\{0\}\equiv V_{m+1}(x)\subset V_m(x)\subset \cdots \subset V_1(x)=\K^d
$$
is  called the {\it Lyapunov filtration} associated to $\mathcal A$  at a given point $x\in Y.$ 
 
 The   numbers  $\chi_m<\chi_{m-1}<\cdots
<\chi_2<\chi_1$ given by  assertion (ii) above are called  the {\it  Lyapunov exponents} of the cocycle $\mathcal{A}$ with respect to the current $T.$
  Moreover, we infer from formulas (\ref{e:Lyapunov}) and (\ref{e:Lyapunov_max}) above  that these characteristic numbers measure heuristically the  expansion rate of
  $\mathcal A$ along different vector-directions $v$ and  along leafwise Brownian trajectories.
In other words,  the stochastic formulas (\ref{e:Lyapunov})-(\ref{e:Lyapunov_max}) only express the  dynamical  character  of the Lyapunov exponents.



The main purpose of this  work is  to find  a  geometric  interpretation  of these  characteristic  quantities.
 Our approach  consists    in replacing the Brownian trajectories by the more appealing objects, namely, the {\it     unit-speed   geodesic rays.}  These paths   are  parameterized  by  their  length (with respect to the leafwise Poincar\'e metric).   Therefore, we characterize 
 the Lyapunov exponents  in terms of  the  expansion rates of $\mathcal A$ along the geodesic rays.
 %

For this  purpose we need to introduce some more  notation and  terminology.
Denote by
$r\D$ the disc in $\C$ of center $0$ and of radius $r$ with $0<r<1$.   In the
Poincar{\'e} disc $(\D,\omega_P),$ 
$r\D$ is also the disc of center 0 and of radius 
\begin{equation}\label{e:change_radius}
R:=\log{1+r\over 1-r}\cdot
\end{equation}
So, we will also denote by $\D_R$ this disc and by $\partial \D_R$ its  boundary.
Conversely, for each $R>0$  we  denote by $r_R$  the unique number $0<r<1$ satisfying the above equation, 
  that is,  $r_R\D=\D_R .$
 
 Recall from (\ref{e:covering_map}) that
 $(\phi_x)_{x\in X}$ is  a given family of  universal covering maps $\phi_x:\ \D\to L_x$  with $\phi_x(0)=x.$
 For every $x\in X,$  
  the  set of all unit-speed geodesic rays $\omega:\ [0,\infty)\to L_x$ starting at $x$ (that is,    $\omega(0)=x$),     can  be described   by  the family  $(\gamma_{x,\theta})_ {\theta\in [0,1)},$ where
 \begin{equation}\label{eq_geodesics}
 \gamma_{x,\theta}(R):=  \phi_x(e^{2\pi i\theta}r_R),\qquad  R\in\R^+.
 \end{equation}
 The path $\gamma_{x,\theta}$ is called the {\it  unit-speed geodesic ray}   at $x$ with the leaf-direction $\theta.$
 Unless  otherwise  specified, the {\it space of leaf-directions} $[0,1)$ is  endowed with the Lebesgue measure. 
 The space of leaf-directions  is  visibly identified, via the map $[0,1)\ni\theta\mapsto e^{2\pi i\theta},$ with the unit circle $\partial \D$ endowed with the normalized rotation measure.

 In order to state  our main results,  the following   notions of expansion  rates for  cocycles  are  needed.
 
  \begin{definition}\label{D:expansion_rate}\rm
 Let $\mathcal A$ be a $\K$-valued   cocycle and  $R>0$  a  time.

  The {\it expansion rate}  of $\mathcal A$  at a  point $x\in X$ in the leaf-direction $\theta$  at time  $R$ along the vector $v\in \K^d\setminus\{0\}$ 
 is the number
 $$\Ec(x,\theta,v,R):={1\over  R}  \log {\| \mathcal A(\gamma_{x,\theta},R)v   \|\over  \| v\|}.$$ 
 
 The {\it  expansion rate}  of $\mathcal A$   at a  point $x\in X$ in the leaf-direction $\theta$ 
  at time  $R$
 is
 \begin{equation*}\begin{split}
 \Ec(x,\theta,R):= \sup\limits_{v \in \K^d\setminus\{0\}} \Ec(x,\theta,v,R)&=\sup_{v \in \K^d\setminus\{0\} } {1\over  R}  \log {\| \mathcal A(\gamma_{x,\theta},R)v   \|\over  \| v\|}\\
 &=
 {1\over  R}  \log {\| \mathcal A(\gamma_{x,\theta},R)  \|}.
 \end{split}
 \end{equation*}

 Given a $\K$-vector subspace   $\{0\}\not=H\subset \K^d,$ the {\it expansion rate}  of $\mathcal A$  at a  point $x\in X$ at time  $R$
 along the vector space  $H$  
 is the interval $\Ec(x,H,R):=[a,b],$ where
 $$
 a:=  \inf_{v \in H\setminus \{0\}} \int_0^1   \Big ({1\over  R}  \log {\| \mathcal A(\gamma_{x,\theta},R)v   \|\over  \| v\|}\Big) d\theta\ \text{and}\
 b:=  \sup_{v \in H\setminus \{0\}} \int_0^1   \Big ({1\over  R}  \log {\| \mathcal A(\gamma_{x,\theta},R)v   \|\over  \| v\|}\Big) d\theta.
 $$
 \end{definition} 
  Notice  that    $ 
 \Ec(x,\theta,v,R) $ (resp. $
 \Ec(x,\theta,R)$)  expresses  geometrically the  expansion rate (resp. the maximal expansion rate) of the  cocycle 
when one travels along the unit-speed geodesic ray $\gamma_{x,\theta}$ up to time $R.$  
  On the other hand, $ 
 \Ec(x,H,R) $   represents the  smallest closed interval which  contains all numbers
$$\int_0^1   \Big ({1\over  R}  \log {\| \mathcal A(\gamma_{x,\theta},R)v   \|\over  \| v\|}\Big) d\theta,
$$
where  $v$ ranges over $H\setminus \{0\}.$
Note that  the above integral is the average of  the expansion rate  of the  cocycle 
when one travels along the unit-speed geodesic rays along the vector-direction $v\in H$   from $x$  to the  Poincar\'e circle  with radius $R$  and  center $x$ spanned  on $L_x.$ 

We say  that a sequence of  intervals $[a(R),b(R)]\subset \R$ indexed by $R\in\R^+$ converges to  a number $\chi\in \R$ and write $\lim_{R\to\infty} [a(R),b(R)]=\chi,$ if   
  $\lim_{R\to\infty} a(R)= \lim_{R\to\infty} b(R)=\chi.$
  
  Now  we are able  to state the main result.
  
  \begin{theorem}  \label{T:main} {\rm (Main Theorem).}
   Let $(X,\Lc)$ be  a  $\Cc^2$-smooth hyperbolic Riemann surface lamination and $T$  a harmonic current. 
 Let 
  $\mu$ be the measure associated to $T$ by (\ref{E:harmonic_measure}).
Assume, moreover, that  $T$ is  extremal   and that $\mu$  is a  (finite) harmonic  measure.
Consider  a   moderate cocycle
$\mathcal{A}:\ \Omega\times \R^+ \to  \GL(d,\K)  .    $ 
Then there is a leafwise saturated Borel set $Y$  of full $T$-measure 
which satisfies the conclusion of  Theorem \ref{T:VA}
and  the following additional  properties: 
 \begin{itemize}
\item[(i)] Assume that  $\mathcal A$ is  H\"older of order $\alpha<2.$ Then,   for each $1\leq i\leq m$ and  for each $x\in Y,$ there is a  set $G_x\subset [0,1)$ of full Lebesgue measure 
 such that     for each  $v\in V_i(x)\setminus V_{i+1}(x),$ 
\begin{equation}\label{e:Lyapunov_geometric}
 \lim_{R\to\infty}\Ec(x,\theta,v,R)= \chi_i,\qquad \theta\in G_x.
\end{equation}
Moreover, the maximal Lyapunov exponent $\chi_1$ satisfies
 \begin{equation}\label{e:Lyapunov_geometric_max}
 \lim_{R\to\infty}\Ec(x,\theta,R)= \chi_1,\qquad \theta\in G_x.
\end{equation}  
\item[(ii)]  Assume that  $\mathcal{A}$ is  strongly  moderate. Then,   for  each  $1\leq i\leq m$ and each $x\in Y,$ 
\begin{equation}\label{e:Lyapunov_geometric_hard}\lim_{R\to\infty}\Ec(x,H_i(x),R)=\chi_i.
\end{equation} 
  \end{itemize}
Here
$\K^d=\oplus_{i=1}^m H_i(x),$ $x\in Y,$ is  the Oseledec  decomposition given by Theorem \ref{T:VA}   and      $\chi_m<\chi_{m-1}<\cdots
<\chi_2<\chi_1$ are the  corresponding     Lyapunov exponents.
\end{theorem}

 
 Theorem  \ref{T:main} gives a geometric meaning to  
 the stochastic formulas (\ref{e:Lyapunov})--(\ref{e:Lyapunov_max}).
 
Applying Theorem  \ref{T:main}    to the case where $X$ is compact and the cocycle $\mathcal A$ is  $\Cc^2$-differentiable (see Proposition
\ref{P:moderate_criterion} below),  we obtain the following 
 \begin{corollary}\label{C:main}
  Let $(X,\Lc)$ be  a  $\Cc^2$-smooth hyperbolic Riemann surface lamination and $T$  a harmonic current. 
Assume, moreover, that $X$ is  compact and  $T$ is  extremal. 
Consider  a $\Cc^2$-differentiable  cocycle
$\mathcal{A}:\ \Omega\times \R^+ \to  \GL(d,\K)  .    $ 
Then there is a leafwise saturated Borel set $Y$  of full $T$-measure 
which satisfies the conclusions of  Theorem \ref{T:VA}
 as well as those  of Theorem  \ref{T:main}.
 \end{corollary}
 
 Let  $(M,\Lc,E)$ be  a transversally $\Cc^2$-smooth (resp. transversally  holomorphic)    singular foliation  by  Riemann surfaces with the set of singularities $E$
in a Riemannian manifold (resp. Hermitian complex  manifold) $M.$
Consider  a leafwise saturated, compact   set $X\subset M\setminus E$  whose  leaves are all hyperbolic.  
So the  restriction of the foliation  $(M\setminus E,\Lc)$ to $X$ gives an inherited compact
  $\Cc^2$-smooth hyperbolic Riemann lamination $(X,\Lc).$ Moreover,
the   holonomy  cocycle of $(M\setminus E,\Lc)$ induces, by restriction, an inherited  $\Cc^2$-differentiable cocycle on $(X,\Lc)$
(see Example \ref{ex:2} below).
Hence, Corollary \ref{C:main} applies  to  the latter cocycle.   
 In particular, when $(M,\Lc,E)$ is a transversally  holomorphic    singular foliation  on a compact Hermitian complex  manifold $M$ of dimension $n,$   
 the corollary 
applies to the induced  holonomy cocycle of rank $n-1$ associated with every  minimal set  $X$ whose leaves are all  hyperbolic. Here 
    a {\it minimal set} is a  leafwise saturated closed  subset
of $M$   which  contains  no proper subset with this property.  

  We outline  the  strategy of the  proofs of Theorem \ref{T:VA}  and  Theorem \ref{T:main}. 
  In  the previous  work \cite[Theorem 3.7 and  3.11]{NguyenVietAnh1},
we have developed a  general  approach to obtain Oseledec multiplicative  ergodic theorems 
for  general laminations.
  For the proof of  Theorem \ref{T:VA}  we  adapt  this approach  
to the present context of hyperbolic Riemann surface laminations.


 The proof of Theorem \ref{T:main} constitutes the core of this  article.
The proof of its  first part  relies on the theory of  Brownian trajectories  on hyperbolic spaces.
More  concretely, 
some quantitative results on  the boundary behavior of   Brownian trajectories  
by  Lyons \cite{Lyons} and Cranston \cite{Cranston} and on  
the shadow of   Brownian trajectories   by geodesic rays (see, for example, Ancona \cite{Ancona})
are  our main  ingredients.
This,  combined with the H\"older  regularity of the  cocycle,
  allows us to replace a  Brownian  trajectory  by a unit-speed geodesic ray with uniformly  distributed leaf-direction.
Hence,   Part (i) of Theorem \ref{T:main} will follow from  Theorem \ref{T:VA}. 

To establish Part (ii) of Theorem \ref{T:main} we  need  two steps.  
In the  first step we  adapt   to our context 
the  so-called {\it Ledrappier  type  characterization  of Lyapunov  spectrum} which  was introduced in the previous  work \cite{NguyenVietAnh1}.  
 This, combined with the ratio ergodic theorem due to Akcoglu-Sucheston \cite{AkcogluSucheston}, allows us to show that
 a  similar version of  formula 
(\ref{e:Lyapunov_geometric_hard}) holds  when  the expansion  rates in terms of  geodesic rays
are replaced by some  heat diffusions associated  with the  cocycle.

The  second step shows that  the above  heat diffusions can be  approximated by  the expansion  rates. 
To do this we    establish a new  geometric estimate on the heat diffusions (see Lemma  \ref{L:key} below).
In fact, this  delicate  estimate  relies on  the  proof of the geometric Birkhoff ergodic theorem  developed in a previous  joint-work with  Dinh and Sibony \cite{DinhNguyenSibony1}. 
Combining the  two steps, Part (ii) of Theorem \ref{T:main}  follows.

 The  article  is  organized  as  follows.  Section  \ref{S:preparation} sets up the background  of this  work.
 Section  \ref{S:leaf-directions} is devoted to  the proof of Theorem \ref{T:VA}  and  Part (i)  of  Theorem  \ref{T:main}. 
 The first  step in the proof of Part (ii) of     Theorem  \ref{T:main} is  developed in Section \ref{S:expectation}.
 The second  step is given in Section   \ref{S:Proofs}.
  When the  lamination $(X,\Lc)$  and the cocycle $\mathcal A$ arise  from some particular  suspensions, in parallel  to  our  approach there is  another  classical way    to define  Lyapunov  exponents   
 using the geodesic flows (see, for example, \cite{BGM}).
The  last  section is devoted to the proof   
   that,  in this  context, both approaches   give  the same  Oseledec decomposition and  the same Lyapunov  spectrum. Nevertheless, our  method gives more  geometric  properties   than the  other  one.  
Moreover, this  context  corresponds to  a  very special case  of our result when the lamination consists only of a single leaf. 
    The section  is concluded  with some open questions and  remarks.

\smallskip

\noindent{\bf Acknowledgement.}
The author would like to thank Alano Ancona for  interesting  discussions.
He also thanks the referee for carefully reading the paper and for suggestions leading to the improvement of the exposition.
This  work was  partially prepared during the author's visit at the Max-Planck Institute for Mathematics in Bonn.
He would like  to express his  gratitude to this  organization for hospitality and for financial support.
 
 \section{Preparatory results}
 \label{S:preparation}

\subsection{Lamination, hyperbolicity and    cocycle} \label{ss:background} 

Let $X$ be a locally compact space.  A {\it   Riemann surface lamination}     $(X,\Lc)$   is  the  data of  a {\it (lamination)  atlas} $\Lc$ 
of $X$ with (laminated) charts 
$$\Phi_p:\U_p\rightarrow \B_p\times \T_p.$$
Here, $\T_p$ is a locally compact  metric space, $\B_p$ is a domain in $\C$,  $\U_p$ is  an open set in 
$X,$ and  
$\Phi_p$ is  a homeomorphism,  and  all the changes of coordinates $\Phi_p\circ\Phi_q^{-1}$ are of the form
$$x=(y,t)\mapsto x'=(y',t'), \quad y'=\Psi(y,t),\quad t'=\Lambda(t),$$
 where $\Psi,$ $\Lambda$ are continuous  functions and $\Psi$ is  holomorphic in  $y.$
 Moreover, we say  that  $(X,\Lc)$ is  {\it  $\Cc^k$-smooth} for some $k\in\N\cup\{\infty\}$ if
   $\Psi$ is $\Cc^k$-smooth   with
respect to $y,$ and its partial derivatives of any total order $\leq k$ with respect to $y$ and $\bar y$ are jointly continuous
with respect  to $(y,t).$

The open set $\U_p$ is called a {\it flow
  box} and the Riemann surface $\Phi_p^{-1}\{t=c\}$ in $\U_p$ with $c\in\T_p$ is a {\it
  plaque}. The property of the above coordinate changes insures that
the plaques in different flow boxes are compatible in the intersection of
the boxes. Two plaques are {\it adjacent} if they have non-empty intersection.
 
A {\it leaf} $L$ is a minimal connected subset of $X$ such
that if $L$ intersects a plaque, it contains that plaque. So a leaf $L$
is a  Riemann surface  immersed in $X$ which is a
union of plaques. 
 
 \begin{definition}\label{D:hyperbolic}
A leaf $L$  of a  lamination $(X,\Lc)$ is  said to be  {\it hyperbolic} if
it  is a   hyperbolic  Riemann  surface, i.e., it is  uniformized   by 
$\D.$   
The   lamination   is  said to be {\it hyperbolic} if  
  its leaves   are all  hyperbolic. 
\end{definition}
 \smallskip
 
\noindent {\bf Standing Hypothesis.} From now on, we always assume that $(X,\Lc)$ is 
a  $\Cc^2$-smooth  Riemann surface lamination.
 
 \smallskip
 
We denote by  $\Cc(X,\Lc)$ the  space of all functions  $f$ defined and  compactly supported  on $ X$ which are  leafwise  $\Cc^2$-smooth
 and transversally continuous, that is,  
 for each  laminated   chart
$\Phi_p:\U_p\rightarrow \B_p\times \T_p$ and all $m,n\in\N$ with $m+n\leq 2,$
 the derivatives  ${\partial^{m+n}(f\circ\Phi_p^{-1})\over  \partial y^m\partial\bar{y}^n}$ exist and are jointly continuous in $(y,t).$

  When a lamination $(X,\Lc)$ satisfies that $X$ is a   manifold and that the leaves of $\Lc$ are Riemann surfaces immersed in $X$, we say that $(X,\Lc)$ is a {\it foliation}.  Moreover,  $(X,\Lc)$ is called a  {\it transversally  $\Cc^k$-smooth foliation} 
(resp.    {\it transversally  holomorphic foliation}  when $X$ is a complex manifold)  if
there is an atlas $\Lc$ of $X$ with charts 
$$\Phi_i:\U_i\rightarrow \B_i\times \T_i,$$
with $\T_i$  an  open set of some $\R^d$ (resp. an open set of some $\C^d$) such that 
 each above map  $\Psi$ is a diffeomorphism  of class $\Cc^k$ (resp.  a biholomorphic map). 
 
 We say that $(M,\Lc,E)$ is  a {\it singular  foliation} if $M$ is  a  manifold and $E\subset M$ is  a closed subset such that
 $\overline{M\setminus E}=M$ and  $(M\setminus E,\Lc)$ is a  foliation. $E$ is  said to be the  {\it set of singularities}.

 Let $\mathcal A:\ \Omega(X,\Lc)\times \R^+\to \GL(d,\K)$   be  a   map that  satisfies the identity, homotopy and
multiplicative laws in Definition  \ref{D:cocycle}. 
In any chart
$\Phi_p:  \U_p\to \B_p\times \T_p$ with  $\B_p$ simply connected,  consider the map
$A_p:\  \B_p\times \B_p\times \T_p\to\GL(d,\K)$  defined  by
$$
A_p(y,z,t):=\mathcal A(\omega,1),
$$
where  $\omega$ is  any leafwise path  such that $\omega(0)=\Phi_p^{-1}(y,t),$ $\omega(1)=\Phi_p^{-1}(z,t)$ 
and  $\omega[0,1]$ is  contained in the simply connected  plaque   $\Phi_p^{-1}(\cdot,t).$
Now   we are able to   explain  the  last law  in  Definition   \ref{D:cocycle} and single out some new classes
of cocycles.  
\begin{definition}\label{D:local_exp}\rm 
 $A_p$  is  called  the {\it local expression} of  $\mathcal A$ on the  chart  $\Phi_p.$
 
  $\mathcal A$ is called   a {\it cocycle} if its local  expression on each chart is Borel measurable.
 
 Now let $(X,\Lc)$ be  a $\Cc^k$-smooth  hyperbolic Riemann surface lamination for some $k\in\N\cup\{\infty\}$.
 \begin{itemize}
 \item [$\bullet$] $\mathcal A$ is called   a {\it leafwise $\Cc^k$-differentiable cocycle}   if,
for each chart  $\Phi_p,$ the local expression   $A_p$  is   $\Cc^k$-differentiable with respect to  $(y,z).$

 \item [$\bullet$]  $\mathcal A$ is called   a {\it   $\Cc^k$-differentiable cocycle}  if,
for each chart $\Phi_p,$ the local expression   $A_p$  is   $\Cc^k$-differentiable with respect to  $(y,z)$
 and its  partial  derivatives of any  total order $\leq k$
 with respect to $(y,z)$ are jointly continuous   in $(y,z,t).$
 \end{itemize}
  \end{definition}
 
\begin{example}\label{ex:1} \rm A fundamental  example of $\Cc^k$-differentiable  $\R$-valued (resp. $\C$-valued) cocycles is  the holonomy  cocycle
 of a transversally $\Cc^k$-smooth (resp. transversally  holomorphic)  foliation $(X,\Lc)$ by hyperbolic Riemann surfaces 
in a Riemannian  manifold (resp. Hermitian complex  manifold) $X.$ See \cite[Proposition 3.3]{NguyenVietAnh1} for more details.  

A more sophisticated situation   will be  discussed in Example \ref{ex:2} below.
\end{example}
 \subsection{Heat diffusions and harmonic currents versus harmonic  measures}
 \label{ss:heat_diffusions}
     
Let  $(X,\Lc)$ be  a hyperbolic  Riemann surface  lamination.  
     The leafwise Poincar\'e metric
$g_P$  induces  the corresponding 
Laplacian $\Delta$  on leaves   (see \cite{DinhNguyenSibony1}).  
 For  every point  $x\in X$
 consider  the   {\it heat  equation} on $L_x$
 $$
 {\partial p(x,y,t)\over \partial t}=\Delta_y p(x,y,t),\qquad  \lim_{t\to 0+} p(x,y,t)=\delta_x(y),\qquad   y\in L_x,\ t\in \R_+.
 $$
Here   $\delta_x$  denotes  the  Dirac mass at $x,$ $\Delta_y$ denotes the  Laplacian  $\Delta$ with respect to the  variable $y,$
 and  the  limit  is  taken  in the  sense of distribution, that is,
$$
 \lim_{t\to 0+
}\int_{L_x} p(x,y,t) f(y) g_P( y)=f(x)
$$
for  every  smooth function  $f$   compactly supported in $L_x.$   

The smallest positive solution of the  above  equation, denoted  by $p(x,y,t),$ is  called  {\it the heat kernel}. Such    a  solution   exists   because  $(L_x,g_P)$ is
complete and   of bounded  geometry  (see, for example,  \cite{CandelConlon2,Chavel}).  
 The  heat kernel  $p(x,y,t)$  gives  rise to   a one  parameter  family $\{D_t:\ t\geq 0\}$ of  diffusion  operators    defined on bounded measurable functions  on $X$ by
 \begin{equation}\label{e:diffusions}
 D_tf(x):=\int_{L_x} p(x,y,t) f(y) g_P (y),\qquad x\in X.
 \end{equation}
 We record here  the  semi-group property  of this  family: 
\begin{equation}\label{e:semi_group}
D_0=\id\quad\text{and}\quad   D_t \mathbf{1} =\mathbf{1}\quad\text{and}\quad D_{t+s}=D_t\circ D_s \quad\text{for}\ t,s\geq 0,
\end{equation}
where $\mathbf{1}$ denotes the function which is identically equal to $1.$

Using   the map $\phi_x:\  \D\to L_x$ given in  (\ref{e:covering_map}), the following identity relates the diffusion operators in  $L_x$ and those  in   the  Poincar\'e disc $(\D,g_P):$ For   $x\in X$ and for every  bounded measurable  function $f$ defined on $L_x,$  
\begin{equation}\label{e:commutation}
 D_t(f\circ \phi_x)=(D_tf)\circ\phi_x, \qquad \textrm{on $L_x$  for all $t\in\R^+.$} 
\end{equation}
 See \cite[Proposition 2.7]{NguyenVietAnh1} for a proof.

 Now  we arrive  at  two notions of harmonic measures.
 \begin{definition}\label{D:harmonic_measure}\rm 
A positive locally finite  Borel measure  $\mu$ on $X$ is said  to be
{\it quasi-harmonic}  if
$$
\int_X  \Delta u \,d\mu=0
$$ 
 for all  functions  $u\in \Cc(X,\Lc).$
 
 A  quasi-harmonic measure $\mu$ is  said to be {\it harmonic }  if  $\mu$ is finite and 
 $\mu$  is {\it $D_t$-invariant} for all $t\in\R^+,$ i.e, 
$$ \int_X  D_tf d\mu=\int_X fd\mu,  \qquad  f\in \Cc(X,\Lc),\ t\in\R^+.   $$
\end{definition}
 
  Let $\Cc^1(X,\Lc)$ denote the  space   of all forms $h$ of bidegree $(1,1)$   defined  on
leaves  of the lamination and  compactly supported  on $X$  such that $h$ is   leafwise  continuous and transversally continuous, that is,  
 for each  laminated   chart
$\Phi_p:\U_p\rightarrow \B_p\times \T_p,$ the form
   $h\circ\Phi_p^{-1}$ is jointly continuous in $(y,t).$ 
  For each chart $\Phi_p:\U_p\rightarrow \B_p\times \T_p,$ the  complex   structure on $\B_p$ induces  a complex  structure on the leaves of $X.$
Therefore,  the operator $\partial$  and $\bar\partial$  can be  defined  so that they act leafwise on  forms
as in the case of manifolds. 
 So we get easily   that
$\ddbar:\  \Cc(X,\Lc)\to\Cc^1(X,\Lc)$. 
   A form $h\in \Cc^1(X,\Lc)$ is  said to be {\it positive} if its restriction to every plaque
 is  a  positive $(1,1)$-form in the  usual  sense of Lelong.
 \begin{definition}\label{D:harmonic_current}
\rm A {\it   harmonic current} $T$  on the lamination is  a  linear continuous  form
    on $\Cc^1(X,\Lc)$ which  verifies $\ddbar T=0$ in the  weak sense (namely  $T(\ddbar f)=0$ for all $f\in  \Cc(X,\Lc)$), and
which is  positive (namely, $T(h)\geq 0$ for all positive forms $h\in \Cc^1(X,\Lc)$). 
\end{definition}
  
   For 
   the  existence  of    nonzero    harmonic currents,  see  the discussion preceding   Theorem  \ref{T:VA}. 

Recall that a positive finite  measure  $\mu$ on the $\sigma$-algebra of Borel sets in $X$ is  said  to be  {\it ergodic} if for every  leafwise  saturated  measurable set $Z\subset X,$
   $\mu(Z)$ is  equal to either $\mu(X)$ or $0.$   A   harmonic  current $T$ is said to be {\it extremal}
   if  it  is an extremal point  in  the  convex  cone of  all  harmonic  currents. 
The  following result relates the notions of harmonic  measures and  harmonic currents. 
\begin{theorem}\label{thm_harmonic_currents_vs_measures}
Let $(X,\Lc)$ be  a hyperbolic Riemann surface  lamination. \\
(i) If $X$ is  compact, then  each quasi-harmonic    measure is  harmonic.
\\ 
(ii)   The map $T\mapsto  \mu=T\wedge g_P$ which is defined on the  convex  cone of all  harmonic  currents is one-to-one 
and  its image is contained in  the convex  cone of  all quasi-harmonic  measures $\mu$.
If, moreover, $X$ is  compact, then
 this map  is  an one-to-one  correspondence between the  convex  cone of  all harmonic  currents $T$ and  
the convex  cone of all harmonic  measures $\mu$. 
\\
(iii) If $T$ is an  extremal harmonic current and $\mu:=T\wedge g_P$ is  finite, then    $\mu$ is  ergodic.
\end{theorem}
\begin{proof}    Assertion (i) follows from  the theory developed in \cite{Garnett} (see also \cite[Proposition 2.4.2]{CandelConlon2} and   \cite[Theorem  5.7]{DinhNguyenSibony1}
for  more explicit proofs).

The first part of  assertion  (ii)  follows from Definition  \ref{D:harmonic_measure}
and  \ref{D:harmonic_current}. When $X$ is  compact, we know,  
by   \cite[Proposition 5.1]{DinhNguyenSibony1}, that  the map $T\mapsto  \mu=T\wedge g_P$ is an one-to-one correspondence  between the  convex  cone of   harmonic  currents $T$
  and  
the convex  cone of  quasi-harmonic  measures $\mu$. 
This, combined with assertion (i), completes the second part of
 assertion (ii).

To prove  assertion (iii),
 suppose in order to get a contradiction that $\mu$ is  not ergodic. So there is a leafwise saturated Borel set $A\subset X$ such that 
$0<\mu(A)<\mu(X).$ Let $\mu_1:=2\mu|_A$ and  $\mu_2:=2\mu|_{X\setminus  A}.$ So  $\mu={\mu_1+\mu_2\over 2},$ and $\mu_1,$ $\mu_2$ are not  co-linear.  Using the local description of $T$ on  each flow box (see
 \cite[Proposition 2.3 ]{DinhNguyenSibony1}), we can show that both $\mu_1$ and $\mu_2$ are quasi-harmonic  measures.
By the first part of assertion (ii), let $T_1,T_2$  be   harmonic currents such that $\mu_1:= T_1\wedge g_P$ and $\mu_2:= T_2\wedge g_P.$  This, combined  with   
  $\mu= {\mu_1+\mu_2\over 2},$ implies that  $T={T_1+T_2\over 2}$ and $T_1,T_2$  are not co-linear. 
This contradicts the extremality of $T.$
\end{proof}

 \subsection{Wiener  measures}
 \label{ss:Wiener}
  
In this  subsection
 we follow the   expositions  given in Section 2.2, 2.4 and 2.5 in our previous  work \cite{NguyenVietAnh1}, which are,  in turn,  inspired by
 Garnett's theory   of leafwise  Brownian  motion in \cite{Garnett} (see also \cite{Candel2,CandelConlon2}).
 
 We first recall the construction of the Wiener measure $W_0$ on the  Poincar\'e  disc $(\D,g_P).$
 Let $\Omega_0$ be  the space consisting of  all continuous  paths  $\omega:\ [0,\infty)\to  \D$  with $\omega(0)=0.$
A {\it  cylinder  set (in $\Omega_0$)} is a 
 set of the form
$$
C=C(\{t_i,B_i\}:1\leq i\leq m):=\left\lbrace \omega \in \Omega_0:\ \omega(t_i)\in B_i, \qquad 1\leq i\leq m  \right\rbrace.
$$
where   $m$ is a positive integer  and the $B_i$'s are Borel subsets of $\D,$ 
and $0< t_1<t_2<\cdots<t_m$ is a  set of increasing times.
In other words, $C$ consists of all paths $\omega\in  \Omega_0$ which can be found within $B_i$ at time $t_i.$
Let $\Ac_0$ be the  $\sigma$-algebra on $\Omega_0$  generated  by all  cylinder sets.
 For each cylinder  set  $C:=C(\{t_i,B_i\}:1\leq i\leq m)$ as  above, define
\begin{equation}\label{eq_formula_W_x_without_holonomy}
W_x(C ) :=\Big (D_{t_1}(\chi_{B_1}D_{t_2-t_1}(\chi_{B_2}\cdots\chi_{B_{m-1}} D_{t_m-t_{m-1}}(\chi_{B_m})\cdots))\Big) (x),
\end{equation}
where,  $\chi_{B_i}$
is the characteristic function of $B_i$ and $D_t$ is the diffusion operator
given  by  (\ref{e:diffusions}) where  $p(x,y,t)$ therein is replaced by the heat kernel  $\bfp(\xi,\zeta,t)$  of the Poincar\'e disc. 
 It is  well-known that $W_0$ can be   extended   to a unique probability measure on $(\Omega_0,\Ac_0).$
This  is the {\it canonical  Wiener measure}  at $0$  on the Poincar\'e disc.

Let $(X,\Lc)$ be  a hyperbolic Riemann surface lamination endowed  with the leafwise Poincar\'e metric $g_P.$
Recall from  Introduction that   $\Omega:=\Omega(X,\Lc) $  is  the space consisting of  all continuous  paths  $\omega:\ [0,\infty)\to  X$ with image fully contained  in a  single   leaf. This  space  is  called {\it the sample-path space} associated to  $(X,\Lc).$
  Observe that
$\Omega$  can be  thought of  as the  set of all possible paths that a 
Brownian particle, located  at $\omega(0)$  at time $t=0,$ might  follow as time  progresses.
For each $x \in  X,$
let $\Omega_x=\Omega_x(X,\Lc)$ be the  space  of all continuous
leafwise paths  starting at $x$ in $(X,\Lc),$ that is,
$$
\Omega_x:=\left\lbrace   \omega\in \Omega:\  \omega(0)=x\right\rbrace.
$$
For each $x\in X,$    the  following  mapping 
\begin{equation}\label{e:Omega_0_vs_Omega_x}
\Omega_0\ni\omega \mapsto  \phi_x\circ\omega\quad
 \text{maps}\quad  \Omega_0\quad\text{bijectively onto}\quad \Omega_x,
 \end{equation}
 where   $\phi_x:\D\to L_x$ is given in (\ref{e:covering_map}).   
Using this bijection  
we obtain   a  natural $\sigma$-algebra    $ \Ac_x$ on  the  space $\Omega_x,$ and  a natural    probability (Wiener) measure $W_x$  on $\Ac_x$  as follows:
\begin{equation}\label{e:W_x}
 \Ac_x:=\{ \phi_x\circ A:\  A\in \Ac_0\}\quad\text{and}\quad    W_x(\phi_x\circ A):=W_0(A),\qquad   A\in  \Ac_0,
\end{equation}
   where $\phi_x\circ A:= \{\phi_x\circ \omega:\ \omega\in A \}\subset \Omega_x.$

For any   function $F\in L^1(\Omega_x,\Ac_x,W_x),$
 the {\it  expectation} of $F$ at $x$ is  the  number
\begin{equation}\label{e:expectation}
\E_x[F]:=\int_{\Omega_x} F(\omega)dW_x(\omega).
\end{equation}
  It is well-known (see \cite[Proposition C.3.8]{CandelConlon2}) that for any  measurable  bounded function $f$ on $L_x,$
\begin{equation}\label{e:expectation_vs_diffusion}
\E_{x}[f(\bullet(t))]=(D_tf)(x),\qquad  t\in\R^+,
\end{equation} 
where $f(\bullet(t))$ is  the  function  given by $\Omega\ni\omega\mapsto f(\omega(t)).$
 
 \subsection{Specialization and several  classes of  cocycles}
 \label{ss:moderate_cocycle}

First we recall some notions and results from \cite[Section 9.1]{NguyenVietAnh1}.
 Fix a  point $x\in X$   and let
   $\phi_{x}:\ \D\to L=L_{x}$   be the universal covering map    given in (\ref{e:covering_map}).    
  We focus   on the leaf $L$  and consider the following {\it projectivization} of $\mathcal A:$ 
  \begin{equation}\label{e:projectivization}
\mathcal A(\omega,t)u:=\left\lbrack  \mathcal A(\omega,t)\tilde u \right\rbrack \ \text{and}\    \|\mathcal A(\omega,t)u\|:={\| \mathcal A(\omega,t)\tilde u \|\over \|\tilde  u\|}\ \text{for}\ t\in\R^+\ \text{and}\  u\in \P^{d-1}(\K),
   \end{equation}
 where  $\tilde u$ is any element in $\K^d\setminus\{0\}$   such that  $u=[\tilde u] .$
  Here $[\cdot]:\ \K^d\setminus \{0\}\to\P^{d-1}(\K)$  is the canonical  projection.
 For each $u\in\P^{d-1}(\K) ,$ consider  
      the  function  $f_{x,u}:\  \D\to \R$ defined by
\begin{equation}\label{e:specialization}
f_{x,u}(\zeta):= \log \| {\mathcal A}(\phi_x\circ\omega,1)u \|,\qquad  \zeta \in \D,
\end{equation}
where  $\omega\in \Omega_0$ is any path  such that  $\omega(1)=\zeta.$ 
This   definition is  well-defined   because of the homotopy law for $\mathcal A$ (see Definition \ref{D:cocycle}) and of the  simple connectedness  of $\D.$
Following \cite{NguyenVietAnh1},
  $f_{x,u}$ is  said  to be  the {\it specialization}  of $\mathcal A$ at $(x,u).$

By   \cite[identities (9.5) and  (9.8)]{NguyenVietAnh1}, we have that
\begin{equation}\label{e:varphi_n_diffusion}
f_{x,u}(0)=0  \quad\textrm{and}\quad  \E_x[\log {\|\mathcal A(\bullet,t)u\|    }  ] =  
 (D_t f_{x,u})(0),\qquad t\in\R^+,
\end{equation}
where  $(D_t)_{t\in\R^+}$ is the family of  diffusion  operators associated with  $(\D,g_P).$ 

Next, we recall  from \cite {NguyenVietAnh1} two conversion rules  for changing  specializations in the  same leaf. 
For this  purpose  let $y\in L$ and pick $\eta\in \phi_x^{-1}(y).$
Define  $v:=[\mathcal A(\phi_x\circ \omega,1) u],$  where $\omega\in \Omega_0$  is  a  leafwise path with $\omega(1)=\eta.$
As a  consequence of the  multiplicative law in Definition \ref{D:cocycle}, the  first  conversion rule (see \cite[identity (9.6)] {NguyenVietAnh1}) states that 
\begin{equation}\label{e:change_spec}
f_{y,v}(\zeta)=f_{x,u}(\zeta)-f_{x,u}(\eta),\qquad \zeta\in\D.
\end{equation}
 We deduce  from
  (\ref{e:varphi_n_diffusion})-(\ref{e:change_spec})  and  the identity 
$D_p\mathbf{1}=\mathbf{1}$ in (\ref{e:semi_group}) the following    second conversion rule  (see \cite[identity (9.9)]{NguyenVietAnh1})
\begin{equation}\label{e:exp_conversion_rule}
 \E_y[\log {\|\mathcal A(\bullet,t)v\|    }  ] =(D_t f_{x,u})(\eta) -f_{x,u}(\eta).
    \end{equation}    
    
Let $\Delta$ be the  Laplacian on the Poincar\'e disc $(\D,g_P),$ that is, for every function $f\in \Cc^2(\D),$
$$
(\Delta  f) g_P=i\ddbar f \qquad\text{on}\ \D.
$$
For every function $f\in \Cc^1(\D),$ let $|df|_P$ be the length  of  the differential $df$ with respect to $g_P,$
that is,   $|df|_P=|df|\cdot g_P^{-1/2}$ on $\D,$ where $|df|$ denotes the Euclidean norm of $df.$  Let $\dist_P$ denote the Poincar\'e distance on $(\D,g_P).$
Inspired by  Definition 8.3 and 8.4 in Candel \cite{Candel2},  we have the  following
\begin{definition}\label{D:moderate_function} \rm  Let $h$ be a real-valued function defined on   $\D$
and let $c,\alpha>0.$   

 $\bullet$ $h$   is called    {\it  moderate}
with   constant $c$  if
$$
\log |h(y)-h(z)|\leq  c\dist_P(y,z)+c,\qquad  y,z\in \D.
$$ 

 $\bullet$ $h$   is called {\it  H\"older of order $\alpha$}
with constant $c$  if
$$
 |h(y)-h(z)|\leq  c\big(\dist_P(y,z)\big)^\alpha+c,\qquad  y,z\in \D.
$$ 

$\bullet$ $h$   is called {\it  Lipschitz}
with constant $c$  if it is   H\"older of order $1$ with constant $c.$
  \end{definition}
 Notice that our definition  of H\"older  functions is   different  from the classical one  since  we  are only concerned about  the quotient
 ${|h(y)-h(z)|/(\dist_P(y,z))^\alpha}$   when $\dist_P(y,z)$ is  large  enough. 

 Now  we are in the position to  formulate   new classes of cocycles.
 \begin{definition}\label{D:moderate_cocycle}\rm 
 Let  $\mathcal A$  be a cocycle.  For every $(x,u)\in X\times\P^{d-1}(\K),$ let $f_{x,u}$ denote, as usual,  the  specialization of $\mathcal A$ at $(x,u).$ 
  
$\bullet$  $\mathcal A$ is called  {\it  moderate} if there  is a constant $c>0$ 
 such that
 for every $(x,u)\in X\times\P^{d-1}(\K),$   $f_{x,u}$  is a  moderate function with constant $c.$

 $\bullet$ A moderate  cocycle  $\mathcal A$ is called  {\it strongly  moderate}  if it is    leafwise $\Cc^2$-differentiable cocycle  and if there  is   a constant $c>0$ 
 such that
 for every $(x,u)\in X\times\P^{d-1}(\K),$
    $|\Delta f_{x,u}|\leq  c$ on $\D.$ 
 
 $\bullet$   $\mathcal A$ is called  {\it H\"older} if there is  $\alpha>0$ such that  
 for every $(x,u)\in X\times\P^{d-1}(\K),$   $f_{x,u}$  is a  H\"older  function of  order $\alpha.$ 
In this  context we also  say   that $\mathcal A$ is  {\it H\"older of order $\alpha$}.
If, moreover, there is  a constant $c>0$ such that  for every $(x,u)\in X\times\P^{d-1}(\K),$   $f_{x,u}$  is a  H\"older  function of  order $\alpha$  with constant $c,$ then we say that  $\mathcal A$ is  {\it uniformly  H\"older (of order $\alpha$)}. 
 
  $\bullet$  $\mathcal A$ is called  {\it Lipschitz}  (resp. {\it uniformly  Lipschitz}) if $\mathcal A$ is H\"older (resp.  uniformly  H\"older)  of order $1$.
\end{definition}
 
  \begin{remark}\label{r:cocycles} \rm 
 \begin{enumerate}
 
\item As an immediate consequence of Definition \ref{D:moderate_cocycle},
the  class of    H\"older (resp.  uniformly H\"older) cocycles $\mathcal A$ of order $\alpha$ is  increasing
  in $\alpha.$

 \item It is  worthy noting the following difference between  a moderate  cocycle and  a H\"older one. For 
a moderate  cocycle, each  specialization $f_{x,u}$  is a  moderate function with the same
constant $c;$ whereas for a
 H\"older  cocycle of order $\alpha,$ each  specialization $f_{x,u}$  is a  H\"older function of order $\alpha$ with some 
constant $c_{x,u}$ which depends  on $x$ and $u.$  So a   moderate cocycle  need not to be  H\"older, and vice versa.
Clearly,  every uniformly H\"older  cocycle of order $\alpha$ is   H\"older of order $\alpha.$ 
However, using  Definition \ref{D:moderate_function} and  Definition \ref{D:moderate_cocycle}, it can be checked  that a uniformly  H\"older  cocycle is  moderate. 
As a partial converse of the last fact, 
it is  shown in Lemma  
\ref{L:strongly_moderate} below that  a  strongly moderate cocycle is necessarily uniformly Lipschitz.
 
\item
Using formula  (\ref{e:change_spec}),  Definition \ref{D:moderate_cocycle} reduces  to
asking the desired properties of $f_{x,u}$  for only one point $x$ in each leaf 
$L$ of the lamination.
\end{enumerate}
\end{remark}

Strongly moderate and uniformly H\"older cocycles exist in abundance.   Here is   a  simple  sufficient criterion.
 \begin{proposition}
\label{P:moderate_criterion} 
A  $\Cc^2$-differentiable cocycle  $\mathcal A$ on a hyperbolic  Riemann surface  lamination $(X,\Lc)$ with $X$ compact is both
strongly moderate
and uniformly Lipschitz.
 \end{proposition}
 \begin{proof}
 Since $X$ is  compact, we know from  Candel \cite{Candel} that  $g_p$ is transversally  continuous.  
 This, coupled with  the assumption that  $\mathcal A$ is  $\Cc^2$-differentiable,   equality (\ref{e:specialization}) and
formula  (\ref{e:change_spec}), implies that
\begin{equation*}
   |df_{x,u}|_P\leq c\quad\text{ and}\quad |\Delta f_{x,u}|\leq c\quad\text{for a constant $c>0$ independent of $x$ and $u.$} 
\end{equation*}
The bound  on $|df_{x,u}|_P$  yields that  $\mathcal A$ is  uniformly Lipschitz, hence
moderate by Item 2. in Remark \ref{r:cocycles}.
This, coupled with 
the bound on  $\Delta f_{x,u}$  implies that $\mathcal A$ is strongly moderate.
 \end{proof}
 \begin{example}\label{ex:2}\rm  
 Let  $(M,\Lc,E)$ be  a transversally $\Cc^2$-smooth (resp. transversally  holomorphic)   singular foliation  by hyperbolic Riemann surfaces with the set of singularities $E$
in a Riemannian manifold (resp. Hermitian complex  manifold) $M.$
Consider  a leafwise saturated, compact   set $X\subset M\setminus E.$  
So the  restriction of the foliation  $(M\setminus E,\Lc)$ to $X$ gives an inherited   lamination $(X,\Lc).$ Moreover,
the   holonomy  cocycle of $(M\setminus E,\Lc)$ induces, by restriction, an inherited  $\Cc^2$-differentiable cocycle $\mathcal A$
on $(X,\Lc).$
By Proposition \ref{P:moderate_criterion}, $\mathcal A$ is  strongly  moderate and uniformly
 Lipschitz.  
 \end{example}

 \section{Proofs of  Theorem  \ref{T:VA} and  the first part of the Main Theorem}
 \label{S:leaf-directions}

We keep the hypotheses  and  notation of  Theorem  \ref{T:VA}.
In what follows,  let $\mathcal A^+:=\mathcal A$ and $\mathcal A^-:={\mathcal A}^{-1},$  and  write 
$\mathcal A^\pm$ for both $\mathcal A^+$ and $\mathcal A^-.$
 \begin{lemma}\label{L:sup_integrability}
  For every $t\in\R^+$ there is  a constant $c=c_t>0$ such that 
$$\int_{\Omega_x} \sup_{0\leq s\leq t} |\log{\|\mathcal A^\pm(\omega,s)\|}| dW_x(\omega)\leq c \qquad\text{for all}\ x\in X.$$
 \end{lemma}
\begin{proof}
 Assume without loss of generality that $t=1$  and let $x\in X.$ 
 Since  $\mathcal A$ is   moderate,
  there  exists a  constant $c' >0$ independent of $x$  such that, for  every universal covering map $\phi_x:\ \D\to L_x$ given 
 in (\ref{e:covering_map}),  we have that
 $$  |\log {\| \mathcal{A}^{\pm 1}(\phi_x\circ\omega,s)   \|}|
 \leq  \exp\big (c'+c'\dist_P(\omega(s),\omega(0))\big),\qquad  \omega\in\Omega_0,\ s\in\R^+.
$$ 
Recall  from \eqref{e:Omega_0_vs_Omega_x} and  \eqref{e:W_x} that 
$\Omega_0\ni\omega \mapsto  \phi_x\circ\omega\in\Omega_x$ is  a bijection  that induces  $W_x$  from $W_0.$
So
\begin{equation}\label{e:integrability}
\begin{split}
\int_{\Omega_x} \sup_{0\leq s\leq 1}|\log{\|\mathcal A^\pm(\omega,s)\|}| dW_x(\omega)&\leq 
\int_{\Omega_0}\exp  \Big(c'+ c'\cdot \sup_{0\leq s\leq 1} \dist_P(\omega(s),\omega(0))\Big)dW_0(\omega)\\
&=e^{c'}\int_{\Omega_0}\exp  \Big( c'\cdot \sup_{0\leq s\leq 1} \dist_P(\omega(s),\omega(0))\Big)dW_0(\omega)
.
\end{split}
\end{equation}
Moreover, for  every positive-valued  function $f\in L^1( \Omega_0,W_0)$  we  have by Fubini's theorem that
$$
\int_0^\infty W_0\left\lbrace  f(\omega)>r \right\rbrace dr =\int_{\Omega_0} f(\omega)dW_0(\omega).
$$
Applying  the  above  identity  to 
$$f(\omega):=\exp  \Big( c'\cdot \sup_{0\leq s\leq 1} \dist_P(\omega(s),\omega(0))\Big)\quad\text{for}\quad  \omega\in\Omega_0,$$
 it follows that   the right hand side of (\ref{e:integrability}) is  equal to
\begin{equation}\label{e:integrability_bis}
 e^{c'}\int_0^\infty W_0\left\{\omega\in\Omega_0:\ \exp\big ( c'\cdot  \sup_{0\leq s\leq 1}\dist_P(\omega(s),\omega(0))\big )  >r \right \}dr.
\end{equation}
On the other hand,   combining  Lemma 8.16 and Corollary 8.8   in   \cite{Candel2}, we  can show that
 there is  a   constant $c''>0$ such that for all $r\geq 1,$
 $$ W_0\left\{\omega\in\Omega_0:\  \sup_{s\in[0,1]}\dist_P(\omega(s),\omega(0))>r \right \}< c'' e^{-r^2/64}.$$
 This     implies  that the integral   in  (\ref{e:integrability_bis}) is  dominated by
\begin{eqnarray*}
 &\,&1+\int_1^\infty W_0\left\{\omega\in\Omega_0:\ \exp\big ( c'\cdot  \sup_{0\leq s\leq 1}\dist_P(\omega(s),\omega(0))\big )  >r \right \}dr\\
 &=&1+\int_1^\infty W_0\left\{\omega\in\Omega_0:\  \sup_{0\leq s\leq 1}\dist_P(\omega(s),\omega(0))\big )  >{\ln r\over c'} \right \}dr\\
& <& 1+ c''\int_1^\infty e^{-(\ln{r}/8c')^2}dr
<\infty.
\end{eqnarray*}
This, coupled with (\ref{e:integrability}), completes the proof. 
\end{proof}
\begin{remark}\label{R:integrability} \rm 
 The  proof of  the upper  bound  of the right hand side of \eqref{e:integrability} also   shows that for    $c',t>0,$ 
 $$
 \int_{\Omega_0}\exp { \Big(c'+ c'\cdot  \dist_P(\omega(t),\omega(0))\Big)}dW_0(\omega)<\infty.
 $$
\end{remark}

 Now  we  arrive at the
 \\
 {\bf End of the proof of Theorem   \ref{T:VA}.}
 Since $\mu$ is  a finite  measure,  Lemma \ref{L:sup_integrability}, applied to $t=1,$ gives that
$$\int_{x\in X} \Big (\int_{\Omega_x} \sup_{0\leq s\leq 1} |\log{\|\mathcal A^\pm(\omega,s)\|}| dW_x(\omega)\Big) d\mu(x)<\infty.$$
So $\mathcal A$ satisfies the integrability condition stated in  \cite[Theorem 3.7]{NguyenVietAnh1}.
 On the other hand,  by Theorem \ref{thm_harmonic_currents_vs_measures} (iii), we know that
   $\mu$ is  ergodic. Consequently, we may apply   \cite[Corollary 3.8 and Theorem 7.3]{NguyenVietAnh1}  and the theorem follows.
 \hfill $\square$

 To prove the first part of the Main Theorem,  we  need the following result on Brownian motion on the  Poincar\'e disc $(\D,g_P).$
\begin{lemma}\label{L:Ancona}
(i) For  $W_0$-almost every  $\omega\in\Omega_0,$  the limit $\omega(\infty):=\lim_{t\to\infty} \omega(t)$ exists and  is  a point in $\partial \D.$ In this  case  let  $\theta=\theta_\omega$ be the unique number in $[0,1)$ such that $e^{2\pi i\theta}=\omega(\infty),$ and 
denote by $\gamma_\omega$  the unit-speed geodesic  ray which is  the  radius of $\D$ issued from $0$ and  landing at $\omega(\infty).$
 \\
 (ii) For  every Borel set $B\subset\partial \D,$  we have that
$$
W_0\Big(\left\lbrace\omega\in\Omega_0:\ \exists\omega(\infty)\in B\right\rbrace\Big) =  \Leb\{\theta\in[0,1):\ e^{2\pi i\theta}\in B\}.
$$
Here $\Leb$ denotes the Lebesgue measure on $[0,1).$
\\ 
(iii) Let $\rho>1.$   Then, 
for  $W_0$-almost every  $\omega\in\Omega_0,$  there is a constant  $c_\omega>0$  such that
$$
\dist_P(\omega(t),\gamma_\omega(t))\leq c_\omega t^{1/2}(\log t)^\rho\qquad \text{for}\ t>2.
$$
\end{lemma}
Roughly speaking, assertion (i) says that  Brownian trajectories issued from $0$ on the Poincar\'e disc are  shadowed by
the unit-speed geodesic  rays which are  radii of $\D.$   Assertion (ii)  states that these  radii are uniformly distributed  with respect to the normalized rotation
measure on $\partial \D.$ Moreover, assertion (iii)  quantifies the distance, as the time progresses,  between a Brownian trajectory and its corresponding unit-speed geodesic ray.
\begin{proof} Assertion (i) and (ii) are classical, 
see, for example,    Ancona's work \cite[Section 7]{Ancona}.

To prove  assertion (iii) recall  from \cite[Theorem 7.3]{Ancona} that for $W_0$-almost  every $w\in\Omega_0,$
there is  a constant $c'_\omega>0$ such that
$$
\dist_P(\omega(t),\gamma_\omega(\R^+))\leq c'_\omega\log t\qquad \text{for}\ t>2.
$$
Since the   point $\gamma_\omega(s )$ with   $s=s_{\omega,t} :=\dist( \omega(t),0     )$
is the nearest point  in $\gamma_\omega(\R^+)$  with respect to  the point $\omega(t),$  it follows that
\begin{equation*}
\dist_P\left(\omega(t),\gamma_\omega (s)\right)\leq c'_\omega\log t\qquad \text{for}\ t>2.
\end{equation*} 
On the  other hand, 
recall from Lyons' work \cite[pp. 3-4]{Lyons} that  for $W_0$-almost  every $w\in\Omega_0,$ there is  a constant $c'_\omega>0$ such that
$$
|\dist_P(\omega(t), 0) -t|\leq c'_\omega  t^{1/2}(\log t)^\rho\qquad \text{for}\ t>2 .
$$
It is  worthy noting that  Lyons' estimate  relies  on a previous result  of Cranston \cite{Cranston} on the boundary behavior of Brownian trajectories.
This, combined  with the previous  estimate, implies 
that
\begin{eqnarray*}
\dist_P(\omega(t),\gamma_\omega(t))&\leq & \dist_P\left(\omega(t),\gamma_\omega (s)\right)+ \dist_P\left(\gamma_\omega(t),\gamma_\omega (s)\right)\\
&=&\dist_P\left(\omega(t),\gamma_\omega (s)\right)+|\dist_P(\omega(t), 0) -t|\\
&\leq & c'_\omega\log t+c'_\omega  t^{1/2}(\log t)^\rho\qquad \text{for}\ t>2 .
\end{eqnarray*}
Choosing $c_\omega:=c'_\omega(1+(\log 2)^{1-\rho}),$
assertion (iii) follows.
\end{proof}

Now  we are in the  position to  prove  the short    part of the Main Theorem.
In fact, we  are partly inspired  by the proof of \cite[Lemma 6.12]{DeroinDupont},  where  Deroin and Dupont investigate  a particular   cocycle of rank 1. But in  their context,  their  method is  only applicable for  Lipschitz cocycles.  
\\
{\bf Proof of assertion (i)  of Theorem  \ref{T:main}.} We only  give  the proof of equality 
(\ref{e:Lyapunov_geometric})
 since
equality (\ref{e:Lyapunov_geometric_max}) can be demonstrated in the  same  way. The  assumption  of the theorem  allows us to apply Theorem
\ref{T:VA} to the cocycle $\mathcal A.$ Consequently,  we obtain properties (i) and (ii) of   Theorem
\ref{T:VA}.
In the  rest  of the proof we keep  the  notation $Y,$ $m,$ $V_i(x),$ $\chi_i$  introduced  in Theorem
\ref{T:VA}. 
Fix an integer  $1\leq i_0\leq  m$ and fix a point $x_0\in  Y.$ Let $v$ be an arbitrary vector in $V_{i_0}(x_0)\setminus V_{i_0+1}(x_0).$ 
  We need to show that
\begin{equation}\label{e:Lyapunov_geometric_new}
 \lim_{R\to\infty}\Ec(x_0,\theta,R,v)= \chi_{i_0}\qquad\textrm{for $\Leb$-almost every $\theta\in[0,1).$ }
\end{equation}
Let $\phi=\phi_{x_0}:\ \D\to L:=L_{x_0}$    be the universal covering map    given in (\ref{e:covering_map}).
Fix $\rho>1.$
Let $\omega$ be  a  {\it generic} path (in the sense of the  measure $W_0$)  in $\Omega_0.$ So 
 $\omega$ satisfies Lemma \ref{L:Ancona}   (i) and (iii).
Writing  $\gamma:=\gamma_\omega,$ we get  a constant $c_\omega$   such that   
 \begin{equation}\label{e:Ancona}
\dist_P(\omega(R),\gamma(R))\leq  c_\omega R^{1/2} (\log R)^\rho\qquad  \textrm{for $R>2.$}
\end{equation}
On the other hand, by (\ref{e:Lyapunov}),  we also have that
 \begin{equation}
\label{e:Lyapunov_new}
\lim\limits_{R\to \infty} {1\over  R}  \log {\| \mathcal{A}(\phi\circ\omega,R)v   \|\over  \| v\|}  =\chi_{i_0}.    
\end{equation}
 For every $R>2$ let $x_R:=\omega(R)$  and $ y_R:=\gamma(R).$  
By  (\ref{e:specialization}),
we have that
 \begin{equation}\label{e:difference}
  {1\over  R}  \log {\| \mathcal{A}(\phi\circ\omega,R)v   \|\over  \| v\|}
 - {1\over  R}  \log {\| \mathcal{A}(\phi\circ\gamma,R)v   \|\over  \| v\|}=  {1\over  R} (f_{x_0,[v]}(x_R)-f_{x_0,[v]}(y_R) )  .
 \end{equation}
 Since $\mathcal A$ is   H\"older of order $\alpha<2,$ there is a  constant $c$ depending only on $\mathcal A$ and $(x_0,v)$  such that
 the modulus of the right hand  side  of (\ref{e:difference}) is  bounded  by 
 \begin{equation*}
R^{-1} \left (c\big (\dist_P(x_R,y_R)\big)^\alpha +c\right)  . 
 \end{equation*} On the other hand, by (\ref{e:Ancona}), we know that $ \dist_P(x_R,y_R)= \dist_P(\omega(R),\gamma(R))\leq c_\omega R^{1/2}(\log R)^\rho.$ 
Consequently, using  that $\alpha <2,$ it follows that the right hand  side  of (\ref{e:difference})  tends to $0$ as $R\to\infty.$
This, combined with  (\ref{e:difference}) and  (\ref{e:Lyapunov_new}), implies that
$$
\lim\limits_{R\to \infty} {1\over  R}  \log {\| \mathcal{A}(\phi\circ\gamma,R)v   \|\over  \| v\|}  =\chi_{i_0},
\qquad \text{for}\ W_0\text{-almost every}\ \omega\in \Omega_0.
$$
Since  $\phi\circ\gamma=\gamma_{x_0,\theta},$ where   $\theta:=\theta_\omega$ (see Lemma \ref{L:Ancona}) and  $\gamma_{x_0,\theta}$ is given in 
\eqref{eq_geodesics}, the last
equality may be rewritten as
$$
\lim\limits_{R\to \infty} {1\over  R}  \log {\| \mathcal{A}(\gamma_{x_0,\theta}  ,R)v   \|\over  \| v\|}  =\chi_{i_0},
\qquad \text{for}\ W_0\text{-almost every}\ \omega\in \Omega_0\ \text{and for}\ \theta:=\theta_\omega.
$$
Putting  this  together  with  Lemma \ref{L:Ancona} (ii),  (\ref{e:Lyapunov_geometric_new})
follows. The proof of Part (i) of the Main Theorem  is thereby completed.  
\hfill $\square$

 \section{Expectation convergence}
 \label{S:expectation}
  
  \subsection{Statement of the expectation  convergence  and a reduction}\label{SS:reduction}
  
The  main purpose of this  section is to prove  the following expectation convergence which  is a  key  ingredient in   the proof of the  second part  of the Main Theorem.
  \begin{theorem} \label{T:expectation} 
  Let $(X,\Lc)$ be  a hyperbolic Riemann surface lamination and $T$  a harmonic current. 
 Let 
  $\mu$ be the measure associated to $T$ by (\ref{E:harmonic_measure}).
Assume, moreover, that  $T$ is  extremal   and that $\mu$  is a   harmonic  measure.
Consider  a   moderate cocycle
$\mathcal{A}:\ \Omega\times \R^+ \to  \GL(d,\K)  .  $  Let $Y$ be  a leafwise saturated Borel set   of full $T$-measure 
which satisfies the conclusion of  Theorem \ref{T:VA}. Let 
$\K^d=\oplus_{i=1}^m H_i(x),$ $x\in Y,$ be  the Oseledec  decomposition given by Theorem \ref{T:VA}   and      $\chi_m<\chi_{m-1}<\cdots
<\chi_2<\chi_1$  the  corresponding     Lyapunov exponents.
 For each $1\leq i\leq m$ and $n\in\N\setminus\{0\},$  define two maximal and minimal functions    $\Mb_{i,n}, \mb_{i,n}:\ Y\to [-\infty,\infty]$ by 
\begin{equation}\label{e:M_n_and_m_n}
\begin{split}
\Mb_{i,n}(x)&:=\sup_{v\in H_i(x)\setminus \{0\}}{1\over  n}\E_x\left[   \log {\| \mathcal A(\bullet,n)v   \|\over  \| v\|}\right],\qquad x\in Y;
\\
\mb_{i,n}(x)&:=\inf_{v\in H_i(x)\setminus \{0\}}{1\over  n}\E_x\left[   \log {\| \mathcal A(\bullet,n)v   \|\over  \| v\|}\right],\qquad x\in Y,
\end{split}
\end{equation}
  where  $\log {\| \mathcal A(\bullet,n)v   \|\over  \| v\|}$ denotes the function $$\Omega_x\ni\omega\mapsto  \log {\| \mathcal A(\omega,n)v   \|\over  \| v\|}.$$
Then    there is  a   Borel set $Y_0\subset Y$   of full $T$-measure  such that 
 $$
 \lim_{n\to\infty} \left[   \mb_{i,n}(x), \Mb_{i,n}(x)\right] = \chi_i ,\qquad  x\in Y_0 \quad\text{and}\quad  1\leq i\leq m.$$
 \end{theorem}
  The rest of the  section is  devoted  to the  proof of Theorem  \ref{T:expectation}.
  
  We make  the  following  reduction.
  Fix  an index $i_0:$  $1\leq i_0\leq m.$ Choose    measurable maps  $\psi_1,\ldots, \psi_{d_{i_0}}:\ X\to \K^d$ such that
  for all $x\in Y,$ $\{\psi_1(x),\ldots, \psi_{d_{i_0}}(x)\}$ is  an orthonormal  basis of $H_{i_0}(x).$
 Consider  the  cocycle $\mathcal B:\  \Omega\times\R^+\to \GL(d_{i_0},\K)$  defined by
$$
\mathcal  B(\omega,t)v:= \psi_{y}^{-1} \Big (\mathcal A(\omega,t)  (\psi_x(v))\Big), \qquad v\in\K^{d_{i_0}},
$$
 where $x:=\omega(0),$ $y:=\omega(t),$ and for each $x\in Y,$ $\psi_x:\ \K^{d_{i_0}}\to H_{i_0}(x)$ is
 the $\K$-linear isomorphism given by  
$$
\psi_x(v):= \sum_{j=1}^{d_{i_0}}\lambda_j\psi_j(x),\qquad \text{for}\ v:=(\lambda_1,\ldots,  \lambda_{d_{i_0}})\in \K^{d_{i_0}}.
$$
Since  $\psi_x$  preserves the Euclidean norms, we   infer that
the  specialization of $\mathcal A$ at $(x,u)$ for any  $x\in Y$ and  $u\in H_{i_0}(x)$
is  the same as the  specialization of $\mathcal B$ at $(x,\psi_x^{-1}(u)).$
Consequently, the proof of Theorem  \ref{T:expectation} for the cocycle  $\mathcal A$ and  $i=i_0$  reduces to 
the  proof for the cocycle $\mathcal B$ having the unique  Lyapunov exponent $\chi_{i_0}.$ Therefore, in the rest of this  section, we may  assume  without loss  of generality that

{\it  The cocycle $\mathcal A$ possesses a unique  Lyapunov exponent $\chi$ (that is, $m=1$). Moreover, we  will write
$\Mb_n$ (resp. $\mb_n$) instead of the unique maximal function $\Mb_{1,n}$ (resp. the unique  minimal function $\mb_{1,n}$).}


  \subsection{Ledrappier type  characterization of Lyapunov  spectrum}
  
  We recall from   \cite[Section 9.2]{NguyenVietAnh1}  some  results  about  dual spaces (see  also \cite{Walters}).
Let $(X,\Bc(X) ,\mu)$  be  a  probability Borel space, where $X$ is a Hausdorff topological space.
Let $E$ be  a  separable  Banach space with  dual space $E^\ast$ and  let $\langle\cdot,\cdot\rangle$ denote the pairing between $E$ and $E^\ast.$
Let $L^1_\mu(E)$  be  the space of all $\mu$-measurable  functions
$f:\ X\to E$     such that  $\|f\|:= \int_X  \|f(x)\| d\mu(x)<\infty.$
This  is  a  Banach space  with the norm $f\mapsto\|f\|,$  where two functions  $f$ and $g$
are identified  if  $f=g$  $\mu$-almost everywhere.
Let $L_\mu^\infty(E^\ast,E)$  be  the space of all maps $f:\  X\to E^\ast$
for which  the  function  $X\ni x\mapsto \langle f(x),v\rangle$ is  bounded  and  measurable  for each $v\in E,$
where  two such functions $f,$ $g$  are identified  if $X\ni x\mapsto \langle f(x),v\rangle$
and $X\ni x\mapsto \langle g(x),v\rangle$ are equal $\mu$-almost everywhere for every $v\in E.$
This  is  a Banach space   with the norm  
$$\|f \|_\infty :=\esup_{x\in X} \| f(x)\|=  \inf_{Y\in\Bc(X):\ \mu(Y)=1}\sup_{x\in Y}\| f(x)\|,  $$
which is finite by the principle of uniform boundedness.
Consider the  map $  \Lambda:\ L_\mu^\infty(E^\ast,E)\to (L^1_\mu(E))^*,$ given by
$$
(\Lambda\gamma)( f):=\int_X \langle\gamma(x),f(x)\rangle d\mu(x),
$$  
where the map $\gamma:\ X\to E^*$    is in  $L_\mu^\infty(E^\ast,E),$
and the map   $f:\ X\to E$   is in  $L^1_\mu(E).$
By \cite{Bourbaki},  $\Lambda$ is an isomorphism of Banach spaces.
In what follows,  for a  locally compact   metric space $\Sigma,$  we denote by $\Mc(\Sigma)$ the  space of all Radon  measures on $\Sigma$ with mass $\leq 1.$
 
We will be  interested  in the case where  $E:= \Cc(P,\R)$ for a       compact metric space $P.$ 
 So  $\Mc(P)$ is  the closed    unit ball of $E^\ast.$ 
The set
 $L^\infty_\mu(\Mc(P))$  of all   measurable  maps $\gamma:\ X\to \Mc(P)$ is contained in the unit ball of  $L_\mu^\infty(E^\ast,E),$
and is closed with respect to the  weak-star topology  $L_\mu^\infty(E^\ast,E).$ 
Hence, $L^\infty_\mu(\Mc(P))$
is  compact with respect to this  topology.
The    set   $L^\infty_\mu(\Mc(P))$ can be  identified with   a subset of the following space: 
$$
\Mc_\mu(X\times P):=\left\lbrace \lambda\in\Mc(X\times P):\ \lambda\ \text{projects to $\mu$ on $X$}   \right\rbrace.
$$
 via  the map  $L^\infty_\mu(\Mc(P))\ni \nu\mapsto \lambda\in \Mc(X\times P),$ where for $ f\in L^1_\mu(\Cc(P,\R)),$ we have
 \begin{equation}  \label{eq_formula_nu_x}
\int_{X\times P} \langle f(x),u\rangle d\lambda(x,u) =\int_X  \langle f(x),\nu(x)\rangle d\mu (x).
\end{equation}
Here,  $\langle f(x),u\rangle$ denotes the evaluation of the  function $f(x) \in \Cc(P,\R)$  at the point $u\in P,$
and $\langle f(x),\nu(x)\rangle$ denotes the pairing  between $E$ and $E^\ast$ evaluated at $f(x)\in E$ and $\nu(x)\in E^\ast.$

In the  remaining part of the section,  let $(X,\Lc)$ be a hyperbolic Riemann lamination    endowed with a harmonic probability measure $\mu$
which is  ergodic, let  $\mathcal A:\ \Omega(X,\Lc)\times \R^+\to \GL(d,\K)$ be a cocycle  admitting a unique  Lyapunov exponent 
$\chi$ with respect to $\mu.$ Assume  in addition that $\mathcal A$ is  moderate.
Set
   \begin{equation*}
 P=\P:=\P^{d-1}(\K)\qquad  \text{ and }\qquad    \Cc(\P):=\Cc(P,\R).
\end{equation*} 
Consider  the  {\it   cylinder lamination} of $\mathcal A,$  denoted  by $(X_{\mathcal A},\Lc_{\mathcal A}),$ which is   defined 
as follows.
 The  ambient topological space $X_{\mathcal A}$ of the  cylinder lamination is   $X\times \P$ which is independent of $\mathcal A.$
Its leaves are defined  as  follows.
For a  point $(x,u)\in X\times \P$ and  for every   simply connected   plaque  $K$  of $(X,\Lc)$ 
passing through $x,$ we define 
the plaque $\mathcal K$  of $(X\times \P,\Lc_{\mathcal A})  $  passing  through  $(x,U)$  by 
$$
\mathcal K=\mathcal K(K,x,u):=\left\lbrace (y,\mathcal A(\omega,1)u):\ y\in K,\ \omega\in\Omega_x,\ \omega(1)=y,\ \omega[0,1]\subset K   \right\rbrace,
$$  
where  $\mathcal A(\omega,1)u$ is  defined using  (\ref{e:projectivization}).  

Note that
 the projection on the  first factor $\pr_1:\ X\times \P\to X $   maps 
each leaf of  $(X_{\mathcal A},\Lc_{\mathcal A})  =(X\times \P,\Lc_{\mathcal A})  $ onto  each leaf of $(X,\Lc)$  locally homeomorphically.
Therefore,
we endow  each leaf of the cylinder lamination with  the (leafwise) Poincar\'e metric, still denoted by $g_P.$
  The  Laplacian
and  the  one  parameter  family  $\{D_t:\ t\geq 0\}$ of the  diffusion operators  are defined  using the newly-defined metric $g_P.$
Since the local expression of  $\mathcal A$  on flow  boxes  is,   in general,  only measurable,
the  cylinder lamination  $(X\times \P,\Lc_{\mathcal A})$ is  a  {\it  measurable lamination} 
in the  sense of \cite[Definition   2.2]{NguyenVietAnh1}.

For a  positive finite Borel measure $\nu$  on $X_{\mathcal A} =X\times \P$ and $t \in \R^+ ,$
$ D_t \nu$ is the positive finite Borel measure on $X_{\mathcal A}$ (unique in the sense of $\nu$-almost
everywhere) satisfies the following condition
$$
\int_{X_{\mathcal A}} D_t f (x,u)d\nu(x,u) =
\int_{X_{\mathcal A}}f (x,u)d(D_t \nu)(x,u)$$
for every bounded measurable function $f :\  X_{\mathcal A}\to  \R.$
 When $\nu$
is a probability measure, so is $D_t \nu$ (see, for example,
\cite[Remark B.14]{NguyenVietAnh1}).
  
 Recall from \cite[Lemma 9.6]{NguyenVietAnh1} the  following result     
\begin{lemma} \label{lem_D_t_invariant_L_mu} For every $t\geq 0,$ the operators $D_t:\ L^\infty_\mu(\Mc(\P))\to L^\infty_\mu(\Mc(\P))$ 
and  $D_t:\ L^1_\mu(\Cc(\P))\to L^1_\mu(\Cc(\P))$   are  contractions, that is,
 $\|D_t \nu\|_\infty \leq \| \nu\|_\infty$ for $\nu\in L^\infty_\mu(\Mc(\P))$ and 
 $\|D_t f\| \leq \| f\|$ for $f
\in L^1_\mu(\Cc(\P)).$ 
\end{lemma}

\begin{definition}\label{defi_weakly_harmonic_measures_which_is_ergodic}
\rm 
  An element  $\nu$  in
$L^\infty_\mu(\Mc(\P))$
 is  said to be a {\it  $\mathcal A$-weakly harmonic measure}  
   if 
$$\int_{X_{\mathcal A}}  D_1 f d\nu=\int_{X_{\mathcal A}} fd\nu
$$   
for all   functions $f\in L^1_\mu(\Cc(\P)).$  

Denote  by  $\Har_\mu (X_{\mathcal A})$ (or equivalently $\Har_\mu (X\times \P)$ when   $\mathcal A$ is clear  from  the context) the set of all $\mathcal A$-weakly harmonic elements.
This is  a convex subset of $L_\mu(\Mc(\P)).$
Using  Lemma  \ref{lem_D_t_invariant_L_mu}, this  set is  also closed.

An element $\nu\in \Har_\mu(X_{\mathcal A} )$ is  said to be {\it extremal} 
if it is  an extremal point
of    this convex closed cone, that is, if $\nu=t\nu_1+(1-t)\nu_2$ for some  $0< t < 1$ and  $\nu_1,\nu_2\in \Har_\mu(X\times \P),$
then $\nu_1$ and $\nu_2$ are constants  times of $\nu .$ Clearly, if  $ \Har_\mu(X\times \P)\not=\{0\},$ the  set of its nonzero extremal elements  
is  always  nonempty.

  Recall from  \cite[Proposition 9.5]{NguyenVietAnh1}  that every extremal element
 $\nu$ of $ \Har_\mu(X\times \P)$   is also an ergodic measure for the cylinder lamination.
In particular, if  $ \Har_\mu(X\times \P)\not=\{0\},$ there always exists a nonzero $\mathcal A$-weakly harmonic element  which is also ergodic.
 \end{definition}

 Consider  the   functions $\varphi$ and  $\varphi_n:\ X\times \P\to\R$    given by
 \begin{equation}\label{e:varphi_n}
\begin{split} 
\varphi(x,u)&:=  \int_{\Omega_x} \log{\|\mathcal A(\omega,1)u\|} dW_x(\omega),\\
\varphi_n&:={1\over n}\sum_{i=0}^{n-1}D_i\varphi. 
\end{split}
 \end{equation}
   \begin{lemma}\label{L:integrability}
 (i)  The  function
 $\varphi$  belongs to  $L^1_\mu(\Cc(\P)).$
 \\
(ii) For every $n\geq1,$
$$
 \varphi_n(x,u) 
= {1\over n}  \int_{\Omega_x} \log{\|\mathcal A(\omega,n)u\|} dW_x(\omega),\qquad
(x,u)\in X\times \P.$$
(iii) For every $n\geq1,$ $\varphi_n$  belongs to  $L^1_\mu(\Cc(\P)).$ Moreover, it is  a Borel measurable  function from $X\times\P\to\R.$
 \end{lemma}
\begin{proof}
 Observe that for each $\omega\in\Omega_x,$
the map $\P\ni u\mapsto  \log{\|\mathcal A(\omega,1)u\|}$ is continuous.
Moreover,   
$$
|\log{\|\mathcal A^\pm(\omega,1)u\|}| \leq  |{\log\|\mathcal A^\pm(\omega,1)\|}|,\qquad  u\in\P.
 $$
 On the other hand, by  Lemma \ref{L:sup_integrability} we  get a constant $c<0$ such that
$$ \int_{\Omega_x} |{\log\|\mathcal A(\omega,1)\|}| dW_x(\omega)<c,\qquad  x\in X.$$
Putting these together and using that $\mu$ is  a finite measure, 
we may apply the Lebesgue's dominated convergence. Consequently,
 $
\varphi(x,\cdot)$ is  continuous  on $\P$  and $
\|\varphi(x,\cdot)\|_{\Cc(\P)}\leq c$ for  any point $x\in X.$  This proves assertion (i).

Using assertion (i),  we proceed as 
 in the proof of  \cite[Lemma  9.9]{NguyenVietAnh1}. Assertion  (ii) follows.
 
 Arguing as in the proof of assertion (i) and using  assertion (ii) yield  the first part of assertion (iii). 
 The  second one is  an immediate consequence of assertion (ii)  (see also   \cite[Theorem 2.6 (i) and Appendix A.7]{NguyenVietAnh1}).
\end{proof}

The following result  has  been proved in \cite[Lemma 9.10]{NguyenVietAnh1}.
 \begin{lemma}\label{L:existence_harmonic_measures}
 Let $(\nu_n)_{n=1}^\infty\subset  L^\infty_\mu(\Mc(\P))$ be such that for each $n\geq 1$ and  for $\mu$-almost every $x\in X,$  $\nu_n(x)$ is   a  Dirac mass at  some  point $u_n(x)\in \P.$  
\\
(i) Then there is a  subsequence $ ( \nu_{n_j} )_{j=1}^\infty$  such that ${1\over n_j} \sum_{k=0}^{n_j-1}D_k\nu_{n_j}$ converges  weakly  to  a  probability   measure 
 $\nu\in\Har_\mu(X\times \P) .$ In particular, there always exists a  probability ergodic $\mathcal A$-weakly harmonic  measure on $X\times \P.$ 
\\
(ii) Moreover, we have that 
  $$\lim_{j\to\infty}  \int_X \varphi_{n_j}(x,u_{n_j}(x))d\mu(x)= \int_{X\times \P} \varphi d\nu ,$$
  where $\varphi_{n_j}$ and $\varphi$ are given by (\ref{e:varphi_n}).
 \end{lemma}
 The  last result of the  subsection  provides an integral formula  for the Lyapunov exponent $\chi.$
 \begin{theorem}\label{T:Ledrappier}
 For every probability   measure 
 $\nu\in\Har_\mu(X\times \P) $  we have that  $\int_{X\times \P} \varphi d\nu=\chi .$
 \end{theorem}
 \begin{proof} 
Consider  first  the  case  where $\nu$  is  ergodic. 
Since $\mathcal A$ admits a unique  Lyapunov exponent $\chi,$  the  theorem follows from 
   \cite [Theorem 9.22, Part 1) (i)] {NguyenVietAnh1}.

Now  consider  the  general case. 
The Choquet unique  representation  theorem (see, for example, \cite[Theorem 2.6.23]{CandelConlon2}) provides the integral  representation    of $\nu$ as  a convex  combinations of elements of  $K,$ where $K$ denotes all extremal  elements of the closed convex cone $\Har_\mu(X\times \P):$
$$
\nu=\int_K \tau d\rho(\tau),
$$
where  $\rho$ is a  probability measure on $K.$ Therefore, we get that
$$\int_{X\times \P} \varphi d\nu=  \int_{\tau\in K}\Big (\int_{X\times \P} \varphi d\tau\Big) d\rho(\tau)   .
$$ 
On the  other hand, since  $\tau\in K$ is  ergodic,  the  first case  implies that the inner integral on the  right hand  side is equal to $\chi.$
Hence, the right hand side is also   equal to  $\chi.$ This  completes the proof. 
\end{proof}

  \subsection{Proof of   the expectation  convergence} 

 \begin{lemma} \label{L:varphi_n}
(i) For each $n\geq 1,$
$$\Mb_n(x)=\sup_{u\in\P}\varphi_n(x,u)\quad\text{and}\quad
\mb_n(x):=\inf_{u\in\P}\varphi_n(x,u),\qquad x\in X.$$ 
 (ii) For each $n\geq 1,$
$\Mb_n$ and $\mb_n$ are Borel functions and $\mu$-integrable, and for each $x\in X,$   the set
$$
\{u\in\P:  \varphi_n(x,u)=\Mb_n(x)\}\quad\text{and}\quad\{u\in\P:  \varphi_n(x,u)=\mb_n(x)\}
$$
are  nonempty closed subsets.
\\
(iii) $(n+k)\Mb_{n+k}\leq n\Mb_n+kD_n\Mb_k$  and $(n+k)\mb_{n+k}\geq n\mb_n+kD_n\mb_k$  for $k,n\in\N.$
 \end{lemma}
\begin{proof}
Combining  the definition of $\Mb_n$ and $\mb_n$ in (\ref{e:M_n_and_m_n})  and  \eqref{e:expectation}, we get that
\begin{eqnarray*}
 \Mb_n(x)&=&\sup_{u\in\P} {1\over n}  \int_{\Omega_x} \log{\|\mathcal A(\omega,n)u\|} dW_x(\omega),
 \\
\mb_n(x)&=&\inf_{u\in\P} {1\over n}  \int_{\Omega_x} \log{\|\mathcal A(\omega,n)u\|} dW_x(\omega).
\end{eqnarray*}
 Consequently, assertion (i)  follows from Lemma \ref{L:integrability}
 (ii).

   By Lemma \ref{L:integrability} (iii),   $\varphi_n$  belongs to  $L^1_\mu(\Cc(\P))$  and  $\varphi_n$ is  a Borel function.
Fixing a  sequence $(u_j)_{j=1}^\infty\subset \P$ which is dense in $\P,$  and  using the continuity of $\P\ni v\mapsto \varphi_n(x,v)$ for each
$x\in X,$  we  see that
$ \Mb_n(x)= \sup_{j\geq 1}\varphi_n(x,u_j)$ for all $ x\in X.$
Hence, $\Mb_n$ is a Borel function. Similarly, we can show that $\mb_n$ is also a Borel function.
The $\mu$-integrability of $\Mb_n$ and $\mb_n$ follows  from  combining  assertion (i) and 
Lemma \ref{L:integrability} (iii).   We also infer from the  continuity of  $\P\ni v\mapsto \varphi_n(x,v)$ for each
$x\in X$ that  the  two sets  considered in (ii) are nonempty closed. This  implies assertion (ii).
 
 We only give  the proof of the  first   inequality in assertion (iii) since  the second one can be proved  similarly.
 Fix  $x_0\in X$ and $n,k\in \N.$  So we have  to show that  
\begin{equation}\label{e:sub_additivity}
(n+k)\Mb_{n+k}(x_0)\leq n\Mb_n(x_0)+k(D_n\Mb_k)(x_0).
\end{equation}
 Fix  a universal covering  $\phi:\ \D\to L_{x_0}$ with $\phi(0)=x_0$  as in (\ref{e:covering_map}). Let $f_u$  be  the  specialization  of $\mathcal A$ at $(x_0,u).$
 By assertion (ii) let $u_0\in\P$ be  such that $\Mb_{n+k}(x_0)=\varphi_n(x_0,u_0).$ 
By  Lemma   \ref{L:integrability} (ii) we  have  that
$$ p\varphi_p(x,u)=  \E_x[\log {\|\mathcal A(\bullet,p)u\|    }  ],\qquad p\in\N,\ u\in \P.$$  
 Putting this and
(\ref{e:varphi_n_diffusion}) and (\ref{e:semi_group}) together, we may write
\begin{equation}\label{e:estimate_1}
(n+k)\Mb_{n+k}(x_0)=(n+k)\varphi_n(x_0,u_0)= (D_{n+k} f_{u_0})(0)
=(D_nf_{u_0} )(0)+ \big(D_n(D_k f_{u_0}-f_{u_0})\big)(0).
\end{equation}
On the other hand, combining  (\ref{e:M_n_and_m_n})  and  (\ref{e:exp_conversion_rule}),  we infer that
\begin{equation*}
p\Mb_p(x)=\max_{u\in \P}(D_p f_{u})(\xi)- f_{u}(\xi),\qquad p\in\N,\ x\in L_{x_0},\ \xi\in\D\ \textrm{with $x=\phi(\xi).$}
\end{equation*}
Since  $f_{u_0}(0)=0$ by (\ref{e:varphi_n_diffusion}), it follows  that 
\begin{equation}\label{e:estimate_2}  
n\Mb_n(x_0)\geq (D_nf_{u_0})(0),
\end{equation}
 and that $k\Mb_k(\phi(\xi))\geq (D_kf_{u_0})(\xi) - f_{u_0}(\xi),$ $\xi\in\D.$
Applying   (\ref{e:commutation})  to  the  function  $\Mb_k$ and to $\phi:\ \D\to L_{x_0},$
and  using the latter  inequality and the fact that $D_n$ is a positive operator  yields  that  
$$k(D_n\Mb_k)(x_0)= (D_n(k\Mb_k))(\phi(0) )\geq \big(D_n(D_k f_{u_0}-f_{u_0})\big)(0).$$
Inserting  this and  (\ref{e:estimate_2}) into the last line  of (\ref{e:estimate_1}), we  obtain (\ref{e:sub_additivity}). This completes the proof of (iii).
\end{proof}

\begin{lemma}\label{L:Ledrappier} The unique  Lyapunov exponent  $\chi$  of $\mathcal A$  satisfies  
 $$\lim_{n\to\infty}\int_X \Mb_n(x)d\mu(x)=\chi= \lim_{n\to\infty}\int_X \mb_n(x)d\mu(x).$$
\end{lemma}
\begin{proof}
We only  prove  the first equality since the proof of the second one is  similar.
Set $a_n:=  n\int_X \Mb_n(x)d\mu(x)$ for $n\geq 1.$
By  Lemma \ref{L:varphi_n}
 (iii), we get that $(n+k)\Mb_{n+k}\leq n\Mb_n+kD_n\Mb_k.$ Integrating both  sides of this  inequality  and using that
 $\mu$ is  $D_n$-invariant (see  Definition  \ref{D:harmonic_measure}), we obtain that  $a_{n+k}\leq  a_n+a_k.$
 So $\lim_{n\to\infty} {a_n\over n}$ exists and is equal to $a:= \inf_{n\geq 1} {a_n\over n}.$ 
Set $$
\Pi_n(x):=\left\lbrace (x,u)\in X\times \P:\ \varphi_n(x,u)=\Mb_n(x)\right\rbrace,\qquad x\in X.
$$
Since  we know from  Lemma \ref{L:varphi_n} (ii) that $\Mb_n$ is measurable and that $\Pi_n(x)$ is a closed set for each $x\in X,$    we can choose by  \cite[Theorem  III.6]{CastaingValadier}  
a measurable  map  $u_n:\ X\to \P$ such that $(x,u_n(x))\in \Pi_n$ for $\mu$-almost every $x\in X.$ 
For $n\geq 1$  let  $\nu_n\in L^\infty_\mu(\Mc(\P))$ be defined  as follows: for each $x\in X,$ $\nu_n(x)$ is  the Dirac mass at $u_n(x).$   
Next,
  applying Lemma  \ref{L:existence_harmonic_measures} to the sequence $(\nu_n)_{n=1}^\infty$  yields 
    a  subsequence $ ( \nu_{n_j} )_{j=1}^\infty$  such that ${1\over n_j} \sum_{k=0}^{n_j-1}D_k\nu_{n_j}$ converges  weakly  to  a  probability   measure 
 $\nu\in\Har_\mu(X\times \P) .$ 
 Moreover, by  Lemma  \ref{L:existence_harmonic_measures} (ii)  we have that 
  $$\lim_{j\to\infty}  \int_X \varphi_{n_j}(x,u_{n_j}(x))d\mu(x)= \int_{X\times \P} \varphi d\nu .$$
  By  Theorem \ref{T:Ledrappier}, the  right hand side is  equal to $\chi.$ 
  A combination of   Lemma \ref{L:integrability} (ii) and  (\ref{e:M_n_and_m_n}) shows that 
 the  left  hand  side is  equal to  $ \lim_{j\to\infty} a_{n_j}$  which is  $a.$
   Hence, we have shown that $a=\chi,$ which  amounts to  $\lim_{n\to\infty}\int_X \Mb_n(x)d\mu(x)=\chi,$
as  desired. 
\end{proof}
  
Now we  discuss some properties of the operator $D=D_1$ given in (\ref{e:diffusions}).
 Since   $p(x,y,t)\geq 0$ (see \cite{Chavel})  and $\mu$ is $D$-invariant (see Definition  \ref{D:harmonic_measure}), we infer that  $D$ is  a  positive linear  operator acting on  the space $L^1(X,\mu)$ and  
 that  $$\|Df\|_{L^1(X,\mu)}\leq \|f\|_{L^1(X,\mu)}.$$ 
In other words, $D$ is {\it  Markovian}  in the sense of Akcoglu-Sucheston \cite{AkcogluSucheston}. Moreover,
  $$\|Df\|_{L^\infty(X,\mu)}\leq \|f\|_{L^\infty(X,\mu)}.$$
On the other hand, since $T$ is  extremal,  Theorem \ref{thm_harmonic_currents_vs_measures} (iii) says that $\mu$ is  ergodic. Consequently,
applying   \cite[Theorem 2.5.5]{CandelConlon2} (see also \cite{DunfordSchwartz}) 
yields the following
\begin{theorem}\label{T:Ito}
For every $f\in L^1(X,\mu),$ ${1\over n}\sum_{i=0}^{n-1} D_i f$ tends to $\int_X fd\mu$ as $n\to\infty$   $\mu$-almost  everywhere.
\end{theorem}
A sequence  $(f_n)_{n=0}^\infty \subset L^1(X,\mu)$ is  said to be   {\it subadditive}  if $f_{n+k}\leq  f_n+D_n f_k$
for all $n,k\in\N.$
Using   Theorem \ref{T:Ito}  and  the fact that $D$ is  Markovian, we  may restate  Akcoglu--Sucheston ratio ergodic theorem for  subadditive  sequences 
as  follows.
\begin{theorem}\label{T:Akcoglu--Sucheston} {\rm (Akcoglu--Sucheston \cite{AkcogluSucheston})}
Let  $(f_n)_{n=0}^\infty \subset L^1(X,\mu)$  be  a subadditive sequence such that $\gamma:=\inf_{n\geq 1}\int_X {f_n\over n}d\mu >-\infty.$
Then $\lim_{n\to\infty} {f_n\over n} =\gamma$ $\mu$-almost  everywhere.
\end{theorem}

Now  we  arrive at the 
\\
{\bf  End of the proof of Theorem   \ref{T:expectation}.}
By Lemma \ref{L:varphi_n}
 (ii)-(iii)
$(n\Mb_n)_{n=1}^\infty$ and $(-n\mb_n)_{n=1}^\infty$ are subadditive sequences.
By Lemma \ref{L:Ledrappier}  we have that
 $$\lim_{n\to\infty}\int_X \Mb_n(x)d\mu(x)=\chi= \lim_{n\to\infty}\int_X \mb_n(x)d\mu(x).$$
 Consequently, applying  Theorem \ref{T:Akcoglu--Sucheston} to $(n\Mb_n)_{n=1}^\infty$ and $(-n\mb_n)_{n=1}^\infty$
 yields the existence  of a Borel set $Y_0\subset Y$ of full $\mu$-measure such that
 $$
 \lim_{n\to\infty}  \Mb_n(x) =\chi= \lim_{n\to\infty}  \mb_n(x)
 $$
for  every $x\in Y_0.$    
     This  completes the proof in the case of a single Lyapunov
exponent, and hence the general case follows from the reduction made in Subsection   \ref{SS:reduction}.
\hfill $\square$


\section{Proof of the second  part of the Main Theorem}
\label{S:Proofs}


We  begin  this section with  some preparatory results on the  heat diffusions on the Poincar\'e  disc $(\D,g_P).$
In what  follows, for $a\in\D$ and $R>0,$  we  denote by      $\D(a,R)$  the Poincar\'e ball $\{\xi\in \D:\ \dist_P(a,\xi)<R\}.$  
For every $R\in \R,$ let $[R]$ be  the integer part of $R,$ i.e., $[R]=n$ if and only if $n\in\Z$ and  $n\leq  R<n+1.$ 

\begin{lemma}\label{L:D_t_Delta}
Let $f\in\Cc^2(\D)$  be such that  $f,$ $|df|_P$ and $\Delta f$  are moderate functions on $\D.$  Then 
$$
(D_tf)(\xi) - f(\xi)=\int_0^t (D_s \Delta  f) (\xi) ds,\qquad t\in\R^+,\ \xi\in\D.
$$ 
\end{lemma}
\begin{proof} We follow  along the same lines  as  the proofs of  
   Candel in \cite[Proposition 8.11]{Candel2}. Indeed, 
recall from   \cite{Candel2}  the  following  Dynkin's formula (see \cite{Dynkin} or \cite[Theorem C.8.1]{CandelConlon2} for a proof): for every   function $f$ in the space $\Cc^2_0(\D)$ of $\Cc^2$-differentiable  functions on $\D$ with compact  support, it holds that
$$
\E_\xi[f\circ \pi_t]  -f(\xi)=\E_\xi\left\lbrack \int_0^t  (\Delta  f)\circ \pi_s ds\right\rbrack,\qquad t\in\R^+,\ \xi\in\D,
$$
where the projection  $\pi_t:\ \Omega\to X$ is given by  $\pi_t(\omega):=\omega(t),$ $\omega\in\Omega,$ $t\in\R^+.$ 
Using identity (\ref{e:expectation_vs_diffusion}), the above formula  can be rewritten, less stochastically and more analytically,  as follows:  
\begin{equation}\label{e:Dynkin}
(D_tf)(\xi) - f(\xi)=\int_0^t   (D_s \Delta f) (\xi) ds,\qquad t\in\R^+,\ \xi\in\D, 
\end{equation}
where  $f \in \Cc^2_0(\D).$  
So it remains to extend (\ref{e:Dynkin})  to the case where $f$ only  satisfies the growth assumption of  the lemma.

Let $(\xi_k)_{k=0}^\infty$ be a  sequence  of points in $\D$  constructed as follows.
Set $\xi_0:=0$ and $p_0:=0.$ For every $n\geq 1,$   suppose  that we have  already  defined
$\xi_j$ with  $j\leq p_{n-1},$ we  want to construct  an integer $p_n>p_{n-1}$ and  the new points  $\xi_j$ with  $p_{n-1}<j\leq p_n$
as follows. Let $p_n=1+[2\pi e^n]+p_{n-1}.$ Let  $\xi_{p_n}$ be  the unique common  point lying on both the positive real axis  of $\C$
and the   circle    $\partial \D(0,n).$ Consider the $(1+[2\pi e^n])$-sided regular polygon inscribed in the  circle $\partial \D(0,n)$
having  $\xi_{p_n}$ as  a  vertex.
 Let  $ \xi_{p_n-1},\ldots,\xi_{p_{n-1}+1}$ be the remaining  vertices of this  polygon.
So   the Poincar\'e distance between  two  consecutive vertices of the polygon is $\leq  1$ since  the Poincar\'e length of  $\partial \D(0,n)$ is $2\pi e^n.$
Continuing this  process, we obtain  a  sequence  $(\xi_k)_{k=0}^\infty\subset \D.$ Note that  $\dist_P(\xi_{p_n},\xi_{p_{n+1}})=1$ for all $n\in\N.$ 

From  this  construction we make the following  observations:
\begin{itemize}
\item[$\bullet$] the family of balls  $(\D(\xi_k,4))_{k=0}^\infty$ is an open cover of $\D;$
\item[$\bullet$] there is  an integer $N>1$ such that  for every $a\in \D,$   the  cardinal of the set $\{k\in\N:\  a\in \D(\xi_k,8)\}$ is $\leq  N.$
\end{itemize}
In particular,  the family of balls  $(U_k:=\D(\xi_k,8))_{k=0}^\infty$ is locally finite in $\D.$
Fix  a smooth compactly  supported  function $\psi:\ \D(0,8)\to [0,1]$  such that
$\psi =1$ on  $\D(0,4).$ For $k\in\N$ fix  an automorphism $\tau_k$ of $\D$ which  sends $\xi_k$ to  $0.$
Consider  the  sequence of functions  $(\psi_k)_{k=1}^\infty$ defined by
$$
\psi_k:={ \psi\circ \tau_k\over  \sum_{j=0}^\infty   \psi\circ \tau_j}\qquad \textrm{on $U_k$.}  
$$
Using  the  above observations, we  see easily that  $\psi_k$ is a well-defined
function in $\Cc^\infty_0(U_k,[0,1])$ for each $k,$ and  $(\psi_k)_{k=1}^\infty$  is  a  partition of unity subordinate to the    cover
$(U_k)_{k=0}^\infty.$ Moreover, 
  there
is a global bound  $c_1>0$ such that for all $k\in\N,$
 \begin{equation}\label{e:global_bound_c_1}
 |d\psi_k|_P\leq  c_1 \quad\text{and} \quad |\Delta\psi_k|\leq c_1.
 \end{equation}
  For each $k\in\N,$  let $f_k$  be the  function in $ \Cc^2_0(\D)$    defined by
 \begin{equation}\label{e:f_k}
  f_k:=\sum_{j=1}^k\psi_i f.
  \end{equation}
  By  (\ref{e:Dynkin}), we have that
\begin{equation} \label{e:Dynkin_index}
(D_tf_k)(\xi) - f_k(\xi)=\int_0^t   (D_s \Delta f_k) (\xi) ds,\qquad t\in\R^+,\ \xi\in\D.
\end{equation}
On the one hand, 
$f_k\to f$ uniformly on compact subsets of $\D$ as  $k\to\infty$  and  $|f_k|\leq  |f|.$  Therefore, we can show  that 
the left hand  side of (\ref{e:Dynkin_index}) tends to $(D_tf)(\xi) - f(\xi)$ uniformly on compact subsets as $k\to\infty.$

We now  examine  the right hand side of \eqref{e:Dynkin_index}. Using  a holomorphic automorphism of $\D$ sending $\xi$ to $0,$ we may suppose without loss of generality that
$\xi=0.$ Recall  the sample-path space $\Omega_0$ and  the  Wiener measure $W_0$  from Subsection \ref{ss:Wiener}.
As $k\to\infty,$ the functions $\Delta f_k$ converge to $\Delta f$ uniformly on compact sets, hence
$$
\int_0^t \Delta f_k(\omega(s))ds\to \int_0^t\Delta f(\omega(s))ds
$$
for each  path $\omega\in\Omega_0,$ since $\omega[0,t]$ is  compact. Thus, 
$\int_0^t \Delta f_k(\bullet(s))ds$ converge  pointwise to  $ \int_0^t\Delta f(\bullet(s))ds$ in $\Omega_0.$
Each  of the  functions
$$
\Omega_0\ni\omega\mapsto \int_0^t \Delta f_k(\omega(s))ds
$$
is  integrable with respect to  $W_0.$ Now  we show that the  convergence is also  dominated.
Indeed,
we infer easily from (\ref{e:global_bound_c_1}) and (\ref{e:f_k}) 
that
  \begin{equation*}
|\Delta f_k|\leq c_2|f| +|\Delta f|+ c_2|df|_P \quad\textrm{for some $c_2>0$ and for all $k\geq 1.$}
\end{equation*}
This  implies that for every $\omega\in\Omega_0,$
\begin{equation}\label{e:Delta_fk}
 \begin{split}
\big|\int_0^t \Delta f_k(\omega(s))ds\big| &\leq  \int_0^t |\Delta f_k|(\omega(s))ds\\
&\leq  c_2 \int_0^t | f|(\omega(s))ds+\int_0^t |\Delta f|(\omega(s))ds+c_2 \int_0^t |d f|_P(\omega(s))ds.
\end{split}
\end{equation}
Using  the  moderateness of $f,$  $\Delta f$ and  $|df|_P,$   we will show that  each term on the right-hand side
of \eqref{e:Delta_fk} is  integrable  with respect to $W_0.$ Indeed,   the  moderateness of $f$  says that  the  first term is   bounded from above by
$$
 \int_0^t \exp{\big( c'+c'\dist_P(\omega(t),0 )\big)}ds  \quad\textrm{for some $c'>0$ depending only on $f.$}
$$
It has been shown in Remark \ref{R:integrability} that
$$
\int_{\Omega_0}\exp{\big( c'+c'\dist_P(\omega(t),0 )\big)}dW_0(\omega)<\infty,
$$
hence, by Fubini's theorem, 
the function
 $$\Omega_0\ni\omega\mapsto \int_0^t | f|(\omega(s))ds
$$
is integrable with respect to $W_0.$
Similarly,  we  can show that the remaining  two  functions  (of $\omega$) on the right-hand side
of  \eqref{e:Delta_fk} is  integrable  with respect to $W_0.$  
Consequently, by  Lebesgue's dominated convergence,
$$
\E_0\big[ \int_0^t \Delta f_k(\bullet(s))ds  \big]\to \E_0\big[ \int_0^t \Delta f(\bullet(s))ds  \big] \quad\textrm{as}\quad k\to\infty. 
$$
Putting this  together with  \eqref{e:expectation}  and \eqref{e:expectation_vs_diffusion}, it follows  that 
$\int_0^t   (D_s \Delta f_k) (0) ds$ converge to  $\int_0^t   (D_s \Delta f) (0) ds$ as $k\to\infty.$
Thus, the  right-hand side of \eqref{e:Dynkin_index} with $\xi=0$ converges to  $\int_0^t   (D_s \Delta f) (0) ds$ as $k\to\infty.$
This, combined with the  convergence of its left-hand side which has been previously  shown,   completes the proof of the  lemma for $\xi=0,$
and hence for every $\xi\in\D.$

 \end{proof}

 \begin{lemma}\label{L:strongly_moderate}
 Let $\mathcal A$ be a  strongly  moderate cocycle. Then   
  there  is   a constant $c>0$ 
 such that
 for every $(x,u)\in X\times\P(\K^d),$ the  specialization $f:=f_{x,u}$ satisfies  the following two conditions:
\begin{itemize}
\item [$\bullet$] both $f$ and $|df|_P$  are  moderate functions  with constant $c;$
\item [$\bullet$]    $|\Delta f|\leq  c$ on $\D.$  
\end{itemize} 
Moreover,   $\mathcal A$ is uniformly Lipschitz.
 \end{lemma}
 \begin{proof}
 By  Definition \ref{D:moderate_cocycle}, 
there is a  constant $c_1>0$ such that  
\begin{equation}\label{e:constant_c_1}
f\quad\text{is a  moderate function  with constant $c_1$
and}\quad |\Delta f|\leq  c_1\quad\text{on $\D.$}
\end{equation}
To complete the proof of the lemma, it suffices to show that $|df|_P\leq  c$ for some constant $c>c_1$ large enough.
By Item 3. in Remark \ref{r:cocycles}, it is sufficient to show that $|df(0)|_P\leq  c.$

Fix  an arbitrary $0<r<1.$ By  Riesz representation formula for the disc $\{z\in\C: |z| < r\}$ gives
for $|z| < r,$
\begin{equation}\label{e:Riesz}
f(z) = {1\over 2\pi}\int_{\zeta\in\C:\ |\zeta|<r} \log {r|z-\zeta|\over |r^2 -z\bar\zeta|} (\Delta f) g_P +{1\over 2\pi}\int_{0}^{2\pi} {1-|z/r|^2\over |e^{i\theta}-z/r|^2} f(re^{i\theta})d\theta.
\end{equation}
We deduce from  \eqref{e:constant_c_1}  that there is a constant $c_2>c_1$ depending only on $c_1$ and $r$ such that $|f(z)|<c_2$ and 
 that $|\Delta f(z)|<c_2$
 for all $|z|<r.$
Using this and  performing  the  derivative  of the  right hand side of \eqref{e:Riesz} with respect to $z,$
we obtain  that $|df(0)|_P\leq  c$ for some constant $c>c_2$ depending only on $c_2$ and $r.$     
  \end{proof}

\begin{lemma}\label{L:key}
Let $f\in\Cc^2(\D)$  be such that  both $f$ and  $|df|_P$   are moderate functions on $\D$
and that  $\Delta f$ is bounded on $\D.$
Then for every $R>1,$
$$
\int_0^1 f(r_Re^{2\pi i\theta})d\theta = (D_{[R]}f)(0)  +O(R^{1/2}\sqrt{\log R}),
$$
where $r_R$ is calculated  according to the conversion rule (\ref{e:change_radius})  and  $[R]$ is the integer part of $R.$
\end{lemma}
\begin{proof}
By  Riesz representation formula we have that
\begin{equation}\label{e:Poisson-Jensen}
\int_0^1 f(re^{2\pi i\theta})d\theta -f(0)={1\over 2\pi}\int  \log^+{r\over |\zeta|}(\Delta f) g_P,
\end{equation}
where   
 $\log^+:=\max\{\log,0\}.$  For $R>0$ let
$$M_R:= \int_\D \log^+ \frac{r}{|\zeta|}g_P=\int_\D \log^+ \frac{r}{|\zeta|} \frac{2}{(1-|\zeta|^2)^2}
id\zeta\wedge d\overline\zeta.$$
Recall   from the proof of  \cite[Lemma 7.6]{DinhNguyenSibony1} that  there is a  constant $c>0$ such that  the following   estimate holds
$$
\left |{1\over  M_R}\int  \log^+{r\over |\zeta|} ug_P -{2\pi\over M_R}\int_0^{M_R\over 2\pi} (D_tu)(0)dt
\right|\leq cR^{-1/2}\sqrt{\log R}\|u\|_\infty
$$
for all $R\in\R^+$ and all bounded  measurable  functions $u$ on $\D.$ Since  $\Delta f$  is  bounded,
the  above   inequality, applied to $\Delta f,$ gives that
$$
\left |{1\over  M_R}\int  \log^+{r\over |\zeta|} (\Delta f)g_P -{2\pi\over M_R}\int_0^{M_R\over 2\pi} (D_t\Delta f)(0)dt
\right|\leq cR^{-1/2}\sqrt{\log R}\|\Delta f\|_\infty
$$
Inserting  (\ref{e:Poisson-Jensen}) into the  first term of the left hand  side, we get that
$$
\left |{2\pi\over M_R} \left(\int_0^1 f(re^{2\pi i\theta})d\theta -f(0)\right)-{2\pi\over M_R}\int_0^{M_R\over 2\pi} (D_t\Delta f)(0)dt
\right|\leq cR^{-1/2}\sqrt{\log R}\|\Delta f\|_\infty.
$$
Moreover, a direct computation shows  that $|M_R- 2\pi R|$ is bounded by a constant  and  it is  clear that $|R-[R]|<1.$ Putting  this  together with  the  estimate    
  $\|D_t\Delta f\|_\infty\leq \|\Delta f\|_\infty<\infty$ for all $t\in\R^+,$ we infer from  the  last line  that
  $$
  \left |{1\over R} \left(\int_0^1 f(re^{2\pi i\theta})d\theta -f(0)\right)-{1\over [R]}\int_0^{[R]} (D_t\Delta f)(0)dt
\right|\leq cR^{-1/2}\sqrt{\log R}\|\Delta f\|_\infty.
  $$
  Applying Lemma \ref{L:D_t_Delta} to the second term on the left  hand side yields that
   $$
  \left |{1\over R} \left(\int_0^1 f(re^{2\pi i\theta})d\theta -f(0)\right)-{1\over [R]}\left ((D_{[R]}f)(0)-f(0)\right)
\right|\leq cR^{-1/2}\sqrt{\log R}\|\Delta f\|_\infty.
  $$
  The proof is  thereby completed.  
\end{proof}

Now  we  are in the position  to  complete the proof of the Main Theorem.
\\
{\bf End of the proof of assertion (ii) of Theorem
 \ref{T:main}.}  
 The proof is  divided into two steps.
 Let $Y_0$ be  the Borel  set of full $\mu$-measure  given by  Theorem \ref{T:expectation}. 
  
  \noindent {\bf Step 1:} {\it Identity (\ref{e:Lyapunov_geometric_hard}) (namely, 
$
 \lim_{R\to\infty}\Ec(x,H_i(x),R)= \chi_i
$)
 holds for each $x\in Y_0$ and for each $1\leq i\leq m.$}
 
 Fix  an index $1\leq i_0\leq m$   and a point $x_0\in Y_0.$
  Let  $\phi_{x_0}:\ \D\to L=L_{x_0}$ be the universal covering map given in  (\ref{e:covering_map}). 
For each  vector $v\in H_{i_0}(x_0)\setminus\{0\},$ let $f_v$  be  the  specialization  of $\mathcal A$ at $(x_0,[v])$ (see formula (\ref{e:specialization})).
 By (\ref{e:varphi_n_diffusion}), we have that
 $$ \E_{x_0}\left[\log {\|\mathcal A(\bullet,R)v\|  \over \|v\|  } \right ] =  
 (D_R f_v)(0),\qquad R>0.
$$ 
On the  other hand,  since $x_0\in Y_0,$ Theorem \ref{T:expectation} tells us that
 $$
 \lim_{n\to\infty}{1\over  n}\inf_{v\in H_{i_0}(x_0)\setminus \{0\}}\E_{x_0}\left[   \log {\| \mathcal A(\bullet,n)v   \|\over  \| v\|}\right] = \lim_{n\to\infty}{1\over  n}\sup_{v\in H_{i_0}(x_0)\setminus \{0\}}\E_{x_0}\left[   \log {\| \mathcal A(\bullet,n)v   \|\over  \| v\|}\right] =       \chi_{i_0} .$$
 Therefore, we  deduce  from the last two lines that
 \begin{equation}\label{e:convergence_D}
\lim_{R\to\infty} \inf_{v\in H_{i_0}(x_0)\setminus \{0\}}{1\over  R}  (D_{[R]} f_v)(0) = \chi_{i_0}
=  \lim_{R\to\infty} \sup_{v\in H_{i_0}(x_0)\setminus \{0\}}{1\over  R} (D_{[R]} f_v)(0) .
 \end{equation}
 Since $\mathcal A$ is  strongly moderate, Lemma \ref{L:strongly_moderate} says that $f_v$ satisfies the assumption of   Lemma \ref{L:key}. Consequently, 
using  this lemma  the last estimate implies that
$$
\inf_{v\in H_{i_0}(x_0)\setminus \{0\}}{1\over R}\int_0^1 f_v(r_Re^{2\pi i\theta})d\theta    = \chi_{i_0}
= \sup_{v\in H_{i_0}(x_0)\setminus \{0\}}{1\over R} \int_0^1 f_v(r_Re^{2\pi i\theta})d\theta  .
$$
Hence,    $\Ec(x_0,H_{i_0}(x_0),R)\to\chi_{i_0}$ as $R\to\infty.$   Step  1 is  thereby completed.

  \noindent {\bf Step 2:} {\it  There exists a leafwise saturated  Borel  set $Y\subset X$ of full $\mu$-measure such that 
$
 \lim_{R\to\infty}\Ec(x,H_i(x),R)= \chi_i
$
 for each $x\in Y$ and  each $1\leq i\leq m.$}

 Let  $Y$ be the  saturation  of $Y_0,$  that is, $Y:=\bigcup_{x\in Y_0}L_z.$ Since $Y_0$ is  of full $\mu$-measure, so is $Y.$ 
 By shrinking $Y$ a  little, we may assume that $Y$ is  leafwise saturated Borel set of full $\mu$-measure. 
 Fix  an index $i_0$   and a point $x_1\in Y.$ 
We  need to show that $
 \lim_{R\to\infty}\Ec(x_1,H_{i_0}(x_1),R)= \chi_{i_0}.
$
 Let  $\phi_{x_1}:\ \D\to L=L_{x_1}$ be the universal covering map given in  (\ref{e:covering_map}). 
 Pick a  point  $x_2\in L\cap Y_0.$ Pick $\xi_2\in \phi_{x_1}^{-1}(x_2).$  Fix a  path $\omega\in\Omega_0$  such that $\omega(1) =\xi_2.$
  For each $v\in  H_{i_0}(x_1),$  we set $u_v:=\mathcal A(\phi_{x_1}\circ \omega,1)v\in H_{i_0}(x_2),$
and let $f_{1,v}$  (resp. $f_{2,v}$) be the specialization of $\mathcal A$ at $(x_1,[v])$ (resp. at $(x_2,[u_v])$).
 
Since  $x_2\in Y_0$ and $   H_{i_0}(x_1)\ni v\mapsto u_v\in  H_{i_0}(x_2)$ is an isomorphism, we infer from identity (\ref{e:convergence_D}) applied to $x_2$ that
 \begin{equation}\label{e:convergence_D_new}
\lim_{R\to\infty} \inf_{v\in H_{i_0}(x_1)\setminus \{0\}}{1\over  R}  (D_{[R]} f_{2,v})(0) = \chi_{i_0}
=  \lim_{R\to\infty} \sup_{v\in H_{i_0}(x_1)\setminus \{0\}}{1\over  R} (D_{[R]} f_{2,v})(0) .
 \end{equation}
  Recall from identity (\ref{e:change_spec}) and the expression  of $u_v$ in terms of $v$    that
\begin{equation}\label{e:f_u_x_2}
f_{2,v}(\xi)= f_{1,v}(\xi)- f_{1,v}(\xi_2),\qquad \xi\in\D,\ v\in H_{i_0}(x_1).  
\end{equation}
Inserting this  into  (\ref{e:convergence_D_new}), we get that
\begin{equation*} 
\lim_{R\to\infty} \inf_{v\in H_{i_0}(x_1)\setminus \{0\}}{1\over  R}  (D_{[R]} f_{1,v})(0) = \chi_{i_0}
=  \lim_{R\to\infty} \sup_{v\in H_{i_0}(x_1)\setminus \{0\}}{1\over  R} (D_{[R]} f_{1,v})(0) .
 \end{equation*}
Using this   we argue  as   we  did     from  (\ref{e:convergence_D}) to the  end of the proof of Step 1. Consequently,
we  conclude that   $\Ec(x_0,H_{i_0}(x_1),R)\to\chi_{i_0}$ as $R\to\infty.$
So 
 the last  step and  hence  the proof of the Main Theorem   is  complete.
 \hfill $\square$


\section{Applications and concluding remarks}
  \label{S:applications} 
  
  Let $S$ be   a  compact  Riemann surface of genus $>1,$  $d\geq 1$ an integer, and $\K\in\{\R,\C\}.$
 Let   $\rho:\ \pi_1(S)\to \GL(d,\K)$ be  a representation; 
this  is  the same as   a local system  $H\to S$ over $S$ with fiber $\K^d.$ In fact, it is well-known that  a local system is equivalent
to a vector bundle endowed  with a flat connection.
For $x\in S$ denote by $H_x$ the  fiber at  $x$ of the local system.
Consider $S$ as  a  lamination  consisting of a single leaf and let $\Omega(S)$ be the sample-path space  associated to $S.$ 

For every  $\omega\in\Omega(S)$ and  $t\in \R^+$ and $v\in H_{\omega(0)},$
let $\hol_{\omega,t}v$ be the the image of $v$ in   $H_{\omega(t)}$ by the holonomy map via  parallel transport (with respect to the Gauss-Manin connection)
along  the path $\omega[0,t].$ 
 
We equip the vector bundle  $H\to S$ with a Riemannian (resp. Hermitian) metric $h.$   
An {\it identifier} $\tau$  of $H\to S$ is  a   smooth    map  which  associates to  each point $x\in X$
a linear   isometry $\tau(x):\ H_x \to \K^d,$ that is, a $\K$-linear morphism  
such that 
\begin{equation}\label{e:isometry}
 \|\tau(x)v\|=\|v\|_h,\qquad v\in  H_x,\ x\in S,
 \end{equation}
 where the norm in the left hand side is the Euclidean norm  (see \cite[Section 3.1]{NguyenVietAnh1}).
 The existence of such a map $\tau$ can be proved   using a partition of unity on $S.$
 
Consider  the  map
$\mathcal A:\ \Omega(S)\times\R^+\to \GL(d,\K)$    defined  as  follows.
$$
\mathcal A(\omega,t):=   \tau(\omega(t))\circ (\hol_{\omega,t})(\omega(0))\circ \tau^{-1}(\omega(0)), \qquad \omega\in \Omega,\ t\in\R^+.
$$ 
It  can be checked that  $\mathcal A$ is a  cocycle in the sense of Definition \ref{D:cocycle}. 
We say that 
$\mathcal A$ is  the {\it associated  cocycle } of  the representation $\rho$ and the identifier $\tau.$
Since $\mathcal A$ is  clearly $\Cc^2$-differentiable, we infer from  Proposition \ref{P:moderate_criterion} that
it is strongly  moderate.

On the  other hand, we deduce from  the assumption  on $S$ that
 the Poincar\'e metric $g_P$ on $S$ is  a	nonzero   finite measure.
So   in formula (\ref{E:harmonic_measure}) we  choose $T:=1$ and  hence $\mu=g_P$ is  an ergodic harmonic measure.   

Therefore, we are in the position to  apply Corollary \ref{C:main} to $\mathcal A.$ 
  Consequently,  we obtain the following result which  characterizes the Lyapunov exponents of $\mathcal A$ both
  dynamically and  geometrically.
 \begin{proposition}\label{P:representation}
   Let   $\rho:\ \pi_1(S)\to \GL(d,\K)$ be  a representation as  above and  $\mathcal A$ its associated  cocycle.
 Then  there exist   a number $m\in\N$  together with $m$ integers  $d_1,\ldots,d_m\in \N$  such that
the following properties hold:
\begin{itemize}
\item[(i)] For   each $x\in S$  
 there   exists a  decomposition of $\K^d$  as  a direct sum of $\K$-linear subspaces 
$$\K^d=\oplus_{i=1}^m H_i(x),
$$
 such that $\dim H_i(x)=d_i$ and  $\mathcal{A}(\omega, t) H_i(x)= H_i(\omega(t))$ for all $\omega\in  \Omega_x$ and $t\in \R^+.$  
 For each $1\leq i\leq m$ and each $x\in S,$ let $V_i(x):=\oplus_{j=i}^m H_j(x).$  Set $V_{m+1}(x)\equiv \{0\}.$

 \item[
(ii)]  There   are real numbers 
$$\chi_m<\chi_{m-1}<\cdots
<\chi_2<\chi_1$$
 such that    for  each $x\in S,$ there is a set $F_x\subset \Omega_x$ of full $W_x$-measure  such that for every  $1\leq i\leq m$ and  every  $v\in V_i(x)\setminus V_{i+1}(x)$
 and every  $\omega\in F_x,$
$$
\lim\limits_{t\to \infty, t\in \R^+} {1\over  t}  \log {\| \mathcal{A}(\omega,t)v   \|\over  \| v\|}  =\chi_i.  
$$
Moreover,    for  every $x\in S$  and  for   every  $\omega\in F_x,$ 
$$
\lim\limits_{t\to \infty, t\in \R^+} {1\over  t}  \log {\| \mathcal{A}(\omega,t)  \|}  =\chi_1    
$$
  \item[
(iii)]    For  each $x\in S,$ there is a set $G_x\subset [0,1)$ of full Lebesgue measure  such that 
     equalities  
 (\ref{e:Lyapunov_geometric})-(\ref{e:Lyapunov_geometric_max}) hold for all $\theta\in G_x.$

\item[(iv)]    For  each $x\in S,$
equality (\ref{e:Lyapunov_geometric_hard}) holds. 
\end{itemize}
Here    $\|\cdot\|$  denotes the standard   Euclidean norm of $\K^d.$  
\end{proposition}

The following decomposition  at  each  fiber of the  local  system $H\to S$
$$
H_x:=\oplus_{i=1}^m H_{i,x},  
$$
where $H_{i,x}:=\tau(x)^{-1} H_i(x),$ $x\in S,$ is called the {\it Oseledec decomposition} at $x$
of the representation $\rho.$ 
The set of numbers $\chi_m<\chi_{m-1}<\cdots
<\chi_2<\chi_1$ is called  the {\it Lyapunov spectrum} of $\rho.$  
  The decreasing  sequence  of  subspaces  of $H_x$ given by:
$$
\{0\}\equiv V_{m+1,x}\subset V_{m,x}\subset \cdots \subset V_{1,x}=H_x,
$$ 
where $V_{i,x}:=\tau(x)^{-1} V_i(x),$ $x\in S,$ 
is  called the {\it Lyapunov filtration}    at  $x$ of  $\rho.$ 
Notice  that  the compactness of $S$ and  the requirement (\ref{e:isometry}) imply that
   the Lyapunov spectrum,     the Oseledec decompositions as well as the Lyapunov filtrations  of $\rho$
are, in fact, independent of the choice of any metric $h$  as well as any   identifier $\tau.$ 

Now  we discuss  another   approach  to define  Lyapunov  exponents  of a linear representation
which has been used  by Bonatti,  G\'omez-Mont   and  many others
  (see  \cite{BGM} and the references therein).
This  approach relies on the geodesic flows. 

Let   $\rho:\ \pi_1(S)\to \GL(d,\K)$ be  a representation as  above, and $H\to S$
its  associated   local system. 
Let $T^1S$ be  the unit tangent bundle of $S$ and $\pi:\ T^1S\to S$  the natural projection. 
Each $y\in T^1S$ corresponds, in a  natural way,  to a unique  unit-speed geodesic ray $\gamma_{x,\theta},$ where $x:=\pi(y)$ and $\theta$ is the direction of $y$  at $x.$
 Endow  $T^1S$ with the Liouville measure $\mu.$ 
Under the    identification $y\equiv (x,\theta),$  $\mu$  may be written as the product  of the measure  $g_P$ for  $x\in S$ and the Lebesgue measure
  for  $\theta\in [0,1).$ This is  the product structure of the Liouville measure. 
Moreover, $\mu$ is 
 an invariant measure which is  ergodic  with respect to the  geodesic flow $(g_t)_{t\in\R^+}$ on $S.$
 Using  $\pi,$ we may  view  $H$ as a  local system over $T^1S$ whose  fiber at $y\in T^1S$ is  set to be simply the fiber $H_{\pi(y)}.$
 We make the  following   observation:
For each $y\in T^1S,$  $g_t(y):\ H_y\to H_{g_ty}$ is an invertible linear map between fibers.
Fix  an identifier $\tau$ and  a metric $h$ of $H\to  S$  as above and let   $\mathcal A$  be  the  associated  cocycle  of  the representation $\rho$ and the identifier $\tau.$  
Identifying the fibers of $H\to
T^1S\overset{\pi}{\to}  S$  with 
$\K^d$  using   $\tau,$  we get that
 \begin{equation}\label{e:coincidence}
\mathcal A(\gamma_{x,\theta},t)=g_t(y)
\end{equation}  
for  every  $t\in\R^+$ and  every unit-speed geodesic ray  $\gamma_{x,\theta}$  that represents $y\in T^1S.$  
Let  $\|g_t(y)\|$ be the norm of the linear map $g_t(y).$ 

Since $S$ is compact we see easily that
  $$
  \int_{T^1S} \sup_{t\in[-1,1]}  \|g_t(y)\|d\mu(y)<\infty.
  $$
  By the Oseledec multiplicative ergodic theorem (see \cite{KatokHasselblatt,Oseledec}), there exist numbers  $\lambda_1>\lambda_2>\cdots>\lambda_r,$
  called  {\it Lyapunov exponents}, and  a measurable  $g_t$-invariant decomposition of the bundle
\begin{equation}\label{e:decomposition_BGV}
H_y=\bigoplus_{i=1}^r H_y^{\lambda_i}
\end{equation}
such that  for $\mu$-almost every $y\in T^1S$ and for every $v\in  H_y^{\lambda_i},$ we have  the asymptotic  growth of norm
$$
\lim_{t\to\pm\infty} {1\over t} \log \|g_tv\|=\lambda_i.
$$
 This, combined  with (\ref{e:coincidence}) and  Proposition \ref{P:representation} (iii) and the product structure of the Liouville measure, implies that  
 $\{\lambda_1,\ldots,\lambda_r\}\equiv\{\chi_1,\ldots,\chi_m\}$ and the two Oseledec decompositions (namely,  
 Proposition \ref{P:representation} (i) and (\ref{e:decomposition_BGV}))  are the same. Consequently, we infer the following remarkable property. 
The  subspaces $H_y^{\lambda_i} $ in  the decomposition (\ref{e:decomposition_BGV}) depend only  on $x:=\pi(y);$
in particular, they  are independent of the  direction $\theta$  while  identifying $y$ with $(x,\theta).$

In summary, in this  particular  example,   our approach  and   the other  one 
using the  geodesic flows give the  same  Oseledec  decomposition.
 However,
 our approach   yields a stronger result. Namely,  the Oseledec decomposition is  holonomy  invariant
(see  Proposition \ref{P:representation} (i)), whereas the other  approach  only  tells us that the decomposition  (\ref{e:decomposition_BGV}) is  $g_t$-invariant.
 
We conclude the  article  with some  remarks and open questions.
\begin{remark}\rm
It  seems  interesting to relax  the conditions imposed on Theorem \ref{T:main}. More concretely, we have the
following three open questions.

\noindent {\bf Question 1.} Is  assertion (i) of Theorem \ref{T:main}  still true if the cocycle $\mathcal A$ is  H\"older of order $\alpha\geq 2$ ?

\noindent {\bf Question 2.}
Is  assertion (ii) of Theorem \ref{T:main} still valid if the  strong moderateness in  Definition   \ref{D:moderate_cocycle}  is weakened 
as  follows: a  cocycle 
$\mathcal A$ is called  {\it strongly  moderate}  if it is    leafwise $\Cc^2$-differentiable cocycle  and if there  is   a constant $c>0$ 
 such that
 for every $(x,u)\in X\times\P^{d-1}(\K),$ 
 both  $f_{x,u}$  and $\Delta f_{x,u}$ are  moderate functions  with constant $c.$  
 
\noindent {\bf Question 3.} Can one  apply the result (or at least the approach) developed in this  article  to the holonomy cocycle of the whole regular part
of a  singular holomorphic  foliation by hyperbolic Riemann surfaces ?  See  \cite{DinhNguyenSibony1,DinhNguyenSibony2,DinhNguyenSibony3,Glutsyuk,Neto,NguyenVietAnh2,NguyenVietAnh3}
for a  recent account on singular holomorphic  foliations.
 
We hope to be able to come back some of these issues  in forthcoming  works.
\end{remark}

 \noindent
Vi{\^e}t-Anh Nguy{\^e}n,  
Universit\'e de Lille 1, 
Laboratoire de math\'ematiques Paul Painlev\'e, 
CNRS U.M.R. 8524,  
59655 Villeneuve d'Ascq Cedex, 
France.\\
{\tt Viet-Anh.Nguyen@math.univ-lille1.fr},
{\tt http://www.math.univ-lille1.fr/$\sim$vnguyen}

\end{document}